\documentclass[]{amsart}

\pdfoutput=1

\usepackage[utf8]{inputenc}
\usepackage{amsmath,amsthm,amssymb,blkarray}
\usepackage{amsfonts}
\usepackage{mathrsfs}
\usepackage{graphicx}
\usepackage{url}
\usepackage{array}
\usepackage{float}
\usepackage{multicol}
\usepackage{tabu}
\usepackage[table]{xcolor}
\usepackage{tikz}
\usetikzlibrary{decorations.markings}
\usetikzlibrary{decorations.pathreplacing}
\usetikzlibrary{arrows,shapes,positioning,backgrounds}
\usetikzlibrary{calc,trees}
\usepackage{nameref} 
\usepackage{caption}
\usepackage[labelformat=simple,labelfont={}]{subcaption}

\usetikzlibrary{shapes.geometric}
\usepackage[margin=1in]{geometry}
\usepackage{mathtools}
\usepackage[shortlabels]{enumitem}
\usepackage{color}
\usepackage{manfnt} 
\usepackage{MnSymbol}
\usepackage[breaklinks]{hyperref}
\usepackage{multicol}
\usepackage{harpoon}
\usepackage{subcaption}
\usepackage{stmaryrd}
\usepackage{calc}
\usepackage{underoverlap}
\usepackage{accents}
\usepackage{amsbsy}

\newcommand{\un}{\underline}
\DeclareMathOperator{\supp}{supp}
\DeclareMathOperator{\card}{card}

\newcommand{\llb}{\llbracket}
\newcommand{\rrb}{\rrbracket}
\newcommand{\ralpha}{{\boldsymbol{\alpha}}}
\newcommand{\rbeta}{{\boldsymbol{\beta}}}
\newcommand{\rgamma}{{\boldsymbol{\gamma}}}
\newcommand{\rdelta}{{\boldsymbol{\delta}}}
\newcommand{\rsigma}{{\boldsymbol{{\sigma}}}}
\newcommand{\rtau}{{\boldsymbol{{\tau}}}}
\newcommand{\rphi}{{\boldsymbol{\varphi}}}

\newcommand{\hatralpha}{{\hat{\boldsymbol{\alpha}}}}
\newcommand{\hatrbeta}{{\hat{\boldsymbol{\beta}}}}

\newcommand{\hatrsigma}{{\hat{\boldsymbol{\sigma}}}}
\newcommand{\btheta}{\mathbin{\theta}}
\newcommand{\rx}{{\boldsymbol{x}}}
\newcommand{\ry}{{\boldsymbol{y}}}

\DeclareMathOperator{\diam}{diam}
\DeclareMathOperator{\sig}{sig}
\DeclareMathOperator{\maj}{maj}
\DeclareMathOperator{\med}{med}

\newcommand{\rlambda}{\boldsymbol{\lambda}}
\newcommand{\rkappa}{{\boldsymbol{{\kappa}}}}
\newcommand{\barsig}{{\overline\sig}}

\DeclareMathOperator{\bs}{\mathcal{S}}

\DeclareMathOperator{\B}{\mathcal{B}}

\newcommand{\lfrac}[2]{\genfrac{}{}{1pt}{}{#1}{\makebox[20pt][c]{$\scriptstyle #2$}}}

\makeatletter
\newcommand{\hathat}[1]{%
\begingroup%
  \let\macc@kerna\z@%
  \let\macc@kernb\z@%
  \let\macc@nucleus\@empty%
  \hat{\mathchoice%
{\raisebox{.4ex}{\vphantom{\ensuremath{\displaystyle #1}}}}%
{\raisebox{.4ex}{\vphantom{\ensuremath{\textstyle #1}}}}%
{\raisebox{.3ex}{\vphantom{\ensuremath{\scriptstyle #1}}}}%
{\raisebox{.14ex}{\vphantom{\ensuremath{\scriptscriptstyle #1}}}}%
\smash{\hat{#1}}}%
\endgroup%
}
\makeatother

\definecolor{orange}{RGB}{255,102,0}
\definecolor{ggreen}{RGB}{0,153,0}
\definecolor{darkblue}{RGB}{0,0,255}
\definecolor{purple}{RGB}{153,51,255}
\definecolor{turq}{RGB}{72,209,204}
\definecolor{gray}{RGB}{220,220,220}
\definecolor{orange2}{RGB}{255,100,0}
\definecolor{purple2}{RGB}{159,51,250}
\definecolor{rred}{rgb}{0.9, 0.17, 0.31}
\definecolor{naugreen}{cmyk}{.43,0,.34,.38}
\definecolor{naublue}{cmyk}{1,.72,0,.32}
\definecolor{mediterranean}{cmyk}{.67,0,.08,.3}
\definecolor{rose}{cmyk}{0,1.00,.20,0}
\definecolor{darkorchid}{cmyk}{.6,.9,0,.05}
\definecolor{butterfly}{cmyk}{.95,.59,0,.10}
\definecolor{springgreen}{cmyk}{1.00,0,.70,.02}
\definecolor{darkred}{cmyk}{0,1,1,.5}
\definecolor{nectarine}{cmyk}{0,0.70,1.00,0}
\definecolor{icyblue}{cmyk}{.84,.25,0,.06}
\definecolor{manatee}{rgb}{0.59, 0.6, 0.67}

\newcommand{\textover}[3][l]{%
 \makebox[\widthof{#3}][#1]{#2}%
 }
\makeatletter

\newcommand{\checknextarg}{\@ifnextchar\bgroup{\gobblenextarg}{\ \cdots \ }}
\newcommand{\gobblenextarg}[2]{\underset{\textover[c]{$\scriptstyle #1$}{$\scriptstyle +++$}}{\un{\textover[c]{$#2$}{$s$}}}\@ifnextchar\bgroup{\gobblenextarg}{ \ \cdots }}
\makeatother

\theoremstyle{definition}
\newtheorem{theorem}{Theorem}[section]
\newtheorem{corollary}[theorem]{Corollary}
\newtheorem{lemma}[theorem]{Lemma}
\newtheorem{conjecture}[theorem]{Conjecture}

\newtheorem{proposition}[theorem]{Proposition}
\newtheorem{example}[theorem]{Example}

\pgfdeclarelayer{background}
\pgfsetlayers{background,main}

\tikzset{
my box/.style = {
, line cap = round
, line join = round
}
}

\newcommand{\highlight}[3]{
\path [my box, line width = #1, draw = #2,opacity=.2] #3;
}

\newcommand{\convexpath}[2]{
[
create hullnodes/.code={
\global\edef\namelist{#1}
\foreach [count=\counter] \nodename in \namelist {
\global\edef\numberofnodes{\counter}
\node at (\nodename) [draw=none,name=hullnode\counter] {};
}
\node at (hullnode\numberofnodes) [name=hullnode0,draw=none] {};
\pgfmathtruncatemacro\lastnumber{\numberofnodes+1}
\node at (hullnode1) [name=hullnode\lastnumber,draw=none] {};
},
create hullnodes
]
($(hullnode1)!#2!-90:(hullnode0)$)
\foreach [
evaluate=\currentnode as \previousnode using \currentnode-1,
evaluate=\currentnode as \nextnode using \currentnode+1
] \currentnode in {1,...,\numberofnodes} {
  let
\p1 = ($(hullnode\currentnode)!#2!-90:(hullnode\previousnode)$),
\p2 = ($(hullnode\currentnode)!#2!90:(hullnode\nextnode)$),
\p3 = ($(\p1) - (hullnode\currentnode)$),
\n1 = {atan2(\y3,\x3)},
\p4 = ($(\p2) - (hullnode\currentnode)$),
\n2 = {atan2(\y4,\x4)},
\n{delta} = {-Mod(\n1-\n2,360)}
  in 
{-- (\p1) arc[start angle=\n1, delta angle=\n{delta}, radius=#2] -- (\p2)}
}
-- cycle
}

\begin{document}

\title{Braid graphs in simply-laced triangle-free Coxeter systems are median}
\author{Jillian Barnes, Jadyn V.~Breland, Dana C.~Ernst, Ruth Perry}
 
\address{
Department of Mathematics and Statistics,
Northern Arizona University PO Box 5717,
Flagstaff, AZ 86011
}
\email{jdb546@nau.edu, Dana.Ernst@nau.edu, Ruth.Perry@nau.edu}

\address{
Mathematics Department,
University of California Santa Cruz,
1156 High Street,
Santa Cruz, CA 95064
}
\email{jbreland@ucsc.edu}


\begin{abstract}
Any two reduced expressions for the same Coxeter group element are related by a sequence of commutation and braid moves. Two reduced expressions are said to be braid equivalent if they are related via a sequence of braid moves. Braid equivalence is an equivalence relation and the corresponding equivalence classes are called braid classes. Each braid class can be encoded in terms of a braid graph in a natural way.  In a recent paper, Awik et al.~proved that when a Coxeter system is simply laced and triangle free (i.e., the corresponding Coxeter graph has no three-cycles), the braid graph for a reduced expression is a partial cube (i.e., isometric to a subgraph of a hypercube). In this paper, we will provide an alternate proof of this fact, as well as determine the minimal dimension hypercube into which a braid graph can be isometrically embedded, which addresses an open question posed by Awik et al. For our main result, we prove that braid graphs in simply-laced triangle-free Coxeter systems are median, which is a strengthening of previous results.
\end{abstract}

\maketitle


\section{Introduction}

Every element of a Coxeter group can be written as an expression in the generators and when the number of generators in an expression is minimal (including multiplicity), the expression is said to be reduced. While an element in a Coxeter group may have many different reduced expressions representing it, Matsumoto's Theorem~\cite[Theorem~1.2.2]{Geck2000} states that any two reduced expressions for the same element are related via a sequence of commutations and so-called braid relations. The commutation and braid relations for a Coxeter system (i.e., a Coxeter group together with a distinguished set of generators) are encoded in the corresponding Coxeter graph. If all of the braid relations of a Coxeter system are of length three, we say that the Coxeter system is simply laced. In addition, if the corresponding Coxeter graph does not contain any three cycles, we say that the Coxeter system is triangle free. 

In light of Matsumoto's Theorem, for a fixed element $w$ in a Coxeter group, we define the Matsumoto graph of $w$ to be the graph having vertex set equal to the set of reduced expressions of $w$, where two vertices are connected by an edge if the corresponding reduced expressions are related by a single commutation or braid move. In~\cite{Bergeron2015}, the authors proved that for finite Coxeter systems, every cycle in a Matsumoto graph has even length. This result was extended to arbitrary Coxeter systems in~\cite{Grinberg2017}. In particular, every Matsumoto graph is bipartite.

Matsumoto's Theorem inspires two different equivalence relations on the set of reduced expressions for a group element. Two reduced expressions for the same Coxeter group element are said to be commutation equivalent if we can obtain one from the other via a sequence of commutation moves. Analogously, we define two reduced expressions to be braid equivalent if they are related by a sequence of braid moves. The corresponding equivalence classes are referred to as commutation classes and braid classes, respectively.

Commutation classes have been studied extensively in the literature, often in the context of Coxeter systems of type $A_n$ (i.e., the symmetric group $S_{n+1}$ with adjacent transpositions as the generating set). For example, see \cite{Bedard1999, Denoncourt2016, Elnitsky1997, Gutierres2020, Gutierres2022, Jonsson2009, Tenner2006, Tenner2023}.  In contrast, braid classes have received very little focused attention. However, braid classes have appeared in the work of Bergeron, Ceballos, and Labb\'e~\cite{Bergeron2015} while Zollinger~\cite{Zollinger1994a} provided formulas for the cardinality of braid classes in the case of Coxeter systems of type $A_n$. Fishel et al.~\cite{Fishel2018} provided upper and lower bounds on the number of reduced expressions for a fixed permutation in Coxeter systems of type $A_n$ by simultaneously studying the interaction between commutation and braid classes. In~\cite{ABCE2024}, Awik et al.~initiated a study of the architecture of braid classes in a special class of Coxeter systems. This paper aims to extend these results.

Each braid class can be encoded in terms of a graph in a natural way. The braid graph for a reduced expression is defined to be the graph with vertex set equal to the corresponding braid class, where two vertices are connected by an edge if the corresponding reduced expressions are related via a single braid move.  Note that every braid graph is equal to one of the connected components of the graph obtained by deleting the edges corresponding to commutation moves in the Matsumoto graph for the corresponding group element. In~\cite{ABCE2024}, the authors proved that every reduced expression in a simply-laced Coxeter system has a unique factorization as a product of so-called links, which in turn induces a decomposition of the braid graph into a box product of the braid graphs for each link factor. Moreover, the authors proved that when the Coxeter system is also triangle free, the braid graph for a reduced expression is a partial cube, which are graphs isometric to a subgraph of a hypercube. 

The distance between to two vertices in a hypercube is simply the number of coordinates in which they differ. This metric is inherited by partial cubes, which play an important role in areas such as parallel computing and coding theory, as one of their key features is the ability to efficiently compute distances between vertices.  See~\cite{ovchinnikov2008partial} and the references therein for a thorough survey of partial cubes. Braid graphs in simply-laced triangle-free Coxeter systems provide a wealth of examples of naturally-occurring partial cubes.  

A median graph is a graph in which every three vertices $u$, $v$, and $w$ have a unique median. That is, there is a unique vertex that simultaneously lies on a geodesic between $u$ and $v$, a geodesic between $u$ and $w$, and a geodesic between $v$ and $w$. Every median graph is a partial cube. The concept of median graph originated in the work of Birkhoff and Kiss~\cite{BirkhoffKiss1947} and Avann~\cite{Avann1961}, but the first time that these graphs were explicitly referred to as median appears to have been in a paper by Nebesky~\cite{Nebesky1971} in 1971. Median graphs have a vast literature. Mulder~\cite{Mulder2010} provides a thorough introduction into the structure theory of median graphs and presents an overview of their many applications. 
Also see~\cite{Klavzar1999} for a survey of median graphs that contains over fifty different characterizations. One application of median graphs has been found in recent studies in phylogenetics~\cite{Bandelt2002, Dress1997}.

In this paper, we prove that braid graphs in simply-laced triangle-free Coxeter systems are median, which strengthens the results of~\cite{ABCE2024}. In addition, we provide a combinatorial characterization of the median of any three braid-equivalent reduced expressions. We also describe the geodesic and cycle structure of braid graphs in simply-laced triangle-free Coxeter systems.

This paper is organized as follows. We begin by recalling the necessary definitions and results from graph theory in Section~\ref{sec:graph theory}, which will be utilized in the sections that follow. Section~\ref{sec:Coxeter systems} introduces the basic terminology of Coxeter systems and establishes our notation. In this section, we also formally define braid classes and braid graphs. In Section~\ref{sec:architecture of braid graphs}, we provide a summary of the required results from~\cite{ABCE2024} that describe the combinatorial architecture of braid graphs in simply-laced triangle-free Coxeter systems. One of our main tools is the notion of signature, which allows one to describe a reduced expression using a minimal amount of information. Section~\ref{sec:geodetic structure} contains new results regarding the geodetic structure of braid graphs. We show that each braid move occurs at most once along a geodesic between two braid equivalent reduced expressions. Moreover, any two geodesics between a fixed pair of reduced expressions utilize the same set of braid moves. We take advantage of these results to bound the diameter of a braid graph in terms of the length of the reduced expression. We also obtain a formula for the distance between two reduced expressions in terms of signatures. In Section~\ref{sec:partial cube structure}, we determine the semicubes of a braid graph in simply-laced triangle-free Coxeter systems, as well as the equivalence classes of edges with respect to the Djoković--Winkler relation, in terms of signatures. As a consequence, we obtain a new proof that braid graphs are partial cubes and compute the isometric dimension of a braid graph, the latter of which settles a conjecture from~\cite{ABCE2024}. Next, Section~\ref{sec:cycle structure} investigates the structure of cycles in braid graphs. Section~\ref{sec:median structure} contains our main result (Theorem~\ref{thm:braid graph for link median} and Corollary~\ref{cor:braid graph for reduced exp median}), which states that the braid graph for a reduced expression in a simply-laced triangle-free Coxeter system is a median graph. Further, we describe the interval between two reduced expressions using 
signatures. Consequently, we obtain a combinatorial description of the median operator on a braid graph. We conclude with a list of conjectures and interesting open problems in Section~\ref{sec:closing}.
	
It is worth mentioning that if one replaces each commutation relation with the absence of a relation (i.e., in the notation of Section~\ref{sec:Coxeter systems}, replace $m(s,t)=2$ with $m(s,t)=\infty$), then each braid graph is actually the full Matsumoto graph for the corresponding group element. In particular, all of our results also apply in the context when the Coxeter graph is a complete graph, where each edge is labeled with 3 or $\infty$ and there are no three cycles labeled entirely with 3's.

The work contained in this paper was initiated in the first and fourth authors' master's theses~\cite{Barnes2022, Perry2024}.

\section{Required graph theory}\label{sec:graph theory}

In this section, we introduce the necessary concepts and terminology from graph theory. All of the graphs discussed throughout this paper are assumed to be finite, connected, and simple. 

Let $G$ be a graph. We will denote the vertex set of $G$ as $V(G)$ and the edge set as $E(G)$. The \emph{subgraph induced by $S \subseteq V(G)$}, denoted $G[S]$, is the graph whose vertex set is $S$ and whose edges are all the edges of $G$ that have both endpoints in $S$. A \emph{graph homomorphism} $f:G\to H$ between graphs $G$ and $H$ is a function $f:V(G)\to V(H)$ satisfying $\{f(u),f(v)\} \in E(H)$ whenever $\{u,v\}\in E(G)$. An injective graph homomorphism $f:G\to H$ is called an \emph{embedding} of $G$ into $H$. If, in addition, $\{f(u), f(v)\} \in E(H)$ implies that $\{u,v\} \in E(G)$, then we say that $f$ is an \emph{induced embedding}. If $f$ is an induced embedding, then $G$ is isomorphic to the subgraph of $H$ induced by the image of $f$. 

A \emph{geodesic} between two vertices $u$ and $v$ of $G$ is a shortest path (sequence of vertices connected by edges) between $u$ and $v$. A subset $U$ of vertices is \emph{convex} if it contains all vertices along the geodesics connecting two vertices of $U$. We define the distance between vertices $u$ and $v$ via
\[
d_G(u, v) := \text{the length of any geodesic between $u$ and $v$.}
\]
Note that if the context is clear, we will simply write $d(u,v)$ in place of $d_G(u,v)$. Since $G$ is assumed to be connected, the function $d_G$ makes $V(G)$ into a finite metric space. The \emph{diameter} of $G$ is defined via
\[
\diam(G) := \max\{d(u,v) \mid u,v \in V(G)\}.
\]
In other words, $\diam(G)$ is the length of any maximal length geodesic between any two vertices of $G$. If $d(u,v) = \diam(G)$, then $u$ and $v$ are said to be \emph{diametrical}.

An \emph{isometric embedding} $f:G\to H$ is a function $f:V(G)\to V(H)$ that satisfies $d_G(u, v) = d_H(f(u), f(v))$ for all $u,v\in V(G)$. An \emph{isometric subgraph} of $G$ is a subgraph $H$ of $G$ with the property that the inclusion map $V(G) \hookrightarrow V(H)$ is an isometric embedding. If $f:G\to H$ is an isometric embedding, then $f$ is automatically an induced embedding since preserving distance also preserves adjacency. In this case, the subgraph of $G$ induced by the image of $f$ is an isometric subgraph of $H$. However, the converse is not always true as seen in the example below. 

\begin{example}\label{ex:inducedandnot}
The embedding $f$ depicted in Figure~\ref{fig:inducedembed1} is an induced embedding. However, this map is not an isometric embedding since $d_G(a,e)=4$ while $d_H(f(a),f(e))=2$.  On the other hand, the embedding $g$ shown in Figure~\ref{fig:notinducedembed} is not an induced embedding since $\{g(d),g(a)\} \in E(H)$ while $\{d,a\} \notin E(G)$.
\end{example}

\begin{figure}[!ht]
\centering
\subcaptionbox{\label{fig:inducedembed1}}[.45\linewidth]{
\begin{tikzpicture}[every circle node/.style={draw, circle ,inner sep=1.25pt},scale=1.25]
\def\shiftA{-3.5}
\node [circle] (1) [label=left:$\scriptstyle e$] at (\shiftA,-2){};
\node [circle] (2) [label=left:$\scriptstyle d$] at (\shiftA,-1){};
\node [circle] (3) [label=left:$\scriptstyle c$] at (\shiftA,0){};
\node [circle] (4) [label=left:$\scriptstyle b$] at (\shiftA,1){};
\node [circle] (5) [label=left:$\scriptstyle a$] at (\shiftA,2){};
\node (x) at (-4.25,2) {$G$};
\node (x) at (0,2) {$H$};
\draw [turq,-, very thick] (1) to (2);
\draw [turq,-, very thick] (2) to (3);
\draw [turq,-, very thick] (3) to (4);
\draw [turq,-, very thick] (4) to (5);
\draw [very thick, ->,black] (-3,0) -- (-2,0) node[midway,above]{$f$};
\node [circle] (1) [label=left:$\scriptstyle f(e)$] at (0,-0.707){};
\node [circle] (2) [] at (.707,-1.414){};
\node [circle] (3) [label=left:$\scriptstyle f(d)$] at (-.707,0){};
\node [circle] (4) at (.707,0){};
\node [circle] (5) [label=left:$\scriptstyle f(c)$] at (0,0.707){};
\node [circle] (6) [label=right:$\scriptstyle f(a)$] at (1.414,0.707){};
\node [circle] (7) [label=left:$\scriptstyle f(b)$] at (.707,1.414){};
\draw [black,-, very thick] (1) to (2);
\draw [turq,-, very thick] (1) to (3);
\draw [black,-, very thick] (1) to (4);
\draw [turq,-, very thick] (3) to (5);
\draw [black,-, very thick] (5) to (4);
\draw [turq,-, very thick] (5) to (7);
\draw [turq,-, very thick] (6) to (7);
\draw [black,-, very thick] (6) to (4);
\begin{pgfonlayer}{background}
\highlight{8pt}{turq}{(1.center) to (3.center) to (5.center) to (7.center) to (6.center)}
\end{pgfonlayer}
\end{tikzpicture}}
\hspace{.5cm}
\subcaptionbox{\label{fig:notinducedembed}}[.45\linewidth]{
\begin{tikzpicture}[every circle node/.style={draw, circle ,inner sep=1.25pt},scale=1.25]
\def\shiftA{-3.5}
\node [circle] (1) [label=left:$\scriptstyle e$] at (\shiftA,-2){};
\node [circle] (2) [label=left:$\scriptstyle d$] at (\shiftA,-1){};
\node [circle] (3) [label=left:$\scriptstyle c$] at (\shiftA,0){};
\node [circle] (4) [label=left:$\scriptstyle b$] at (\shiftA,1){};
\node [circle] (5) [label=left:$\scriptstyle a$] at (\shiftA,2){};
\node (x) at (-4.25,2) {$G$};
\node (x) at (0,2) {$H$};
\draw [turq,-, very thick] (1) to (2);
\draw [turq,-, very thick] (2) to (3);
\draw [turq,-, very thick] (3) to (4);
\draw [turq,-, very thick] (4) to (5);
\draw [very thick, ->,black] (-3,0) -- (-2,0) node[midway,above]{$g$};
\node [circle] (1) [label=left:$\scriptstyle g(d)$] at (0,-0.707){};
\node [circle] (2) [label=left:$\scriptstyle g(e)$] at (.707,-1.414){};
\node [circle] (3) [label=left:$\scriptstyle g(c)$] at (-.707,0){};
\node [circle] (4) [label=right:$\scriptstyle g(a)$] at (.707,0){};
\node [circle] (5) [label=left:$\scriptstyle g(b)$] at (0,0.707){};
\node [circle] (6) at (1.414,0.707){};
\node [circle] (7) at (.707,1.414){};
\draw [turq,-, very thick] (1) to (2);
\draw [turq,-, very thick] (1) to (3);
\draw [black,-, very thick] (1) to (4);
\draw [turq,-, very thick] (3) to (5);
\draw [turq,-, very thick] (5) to (4);
\draw [black,-, very thick] (5) to (7);
\draw [black,-, very thick] (6) to (7);
\draw [black,-, very thick] (6) to (4);
\begin{pgfonlayer}{background}
\highlight{8pt}{turq}{(1.center) to (2.center) to (3.center) to (5.center) to (4.center)}
\end{pgfonlayer}
\end{tikzpicture}}
\caption{An induced embedding and an embedding that is not induced as described in Example~\ref{ex:inducedandnot}. Neither embedding is an isometric embedding.}\label{fig:inducedembed} 
\end{figure}
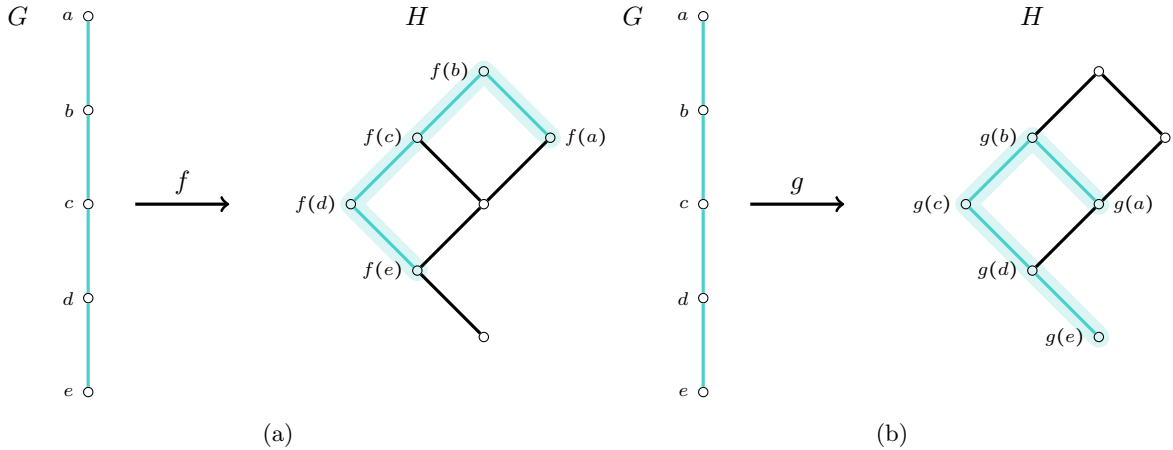

If the subgraph determined by a cycle of a graph is an isometric subgraph, we refer to the cycle and the corresponding subgraph as a \emph{isometric cycle}. If the vertices of a cycle form a convex set, we say that the cycle and the corresponding subgraph are a \emph{convex cycle}.  Note that every convex cycle is also an isometric cycle, but the converse is not true as the next example illustrates.

\begin{example}\label{ex:isometric vs convex}
The subgraph indicated in Figure~\ref{fig:convex cycle} is a convex cycle, and hence an isometric cycle.  On the other hand, the subgraph highlighted in Figure~\ref{fig:isometric cycle} is an isometric cycle but not a convex cycle. The subgraph highlighted in Figure~\ref{fig:cycle} is a cycle that is neither isometric nor convex.
\end{example} 

\begin{figure}[!ht]
\centering
\subcaptionbox{\label{fig:convex cycle}}[.3\linewidth]{
	\begin{tikzpicture}[every circle node/.style={draw, circle ,inner sep=1.25pt},scale=1.25]
		\node [circle] (6) at (4.8,0){};
		\node [circle] (7) [] at (4.093,0.707){};
		\node [circle] (8) [] at (5.507,0.707){};
		\node [circle] (9) [] at (4.8,1.414){};
		\node [circle] (10) [] at (4.8,-1){};
		\node [circle] (11) [] at (4.093,-.293){};
		\node [circle] (12) [] at (5.507,-.293){};
		\node [circle] (13) at (4.8,0.414){};
		\draw [turq,-, very thick] (6) to (7);
		\draw [turq,-, very thick] (6) to (8);
		\draw [turq,-, very thick] (7) to (9);
		\draw [turq,-, very thick] (9) to (8);
		\draw [black,-, very thick] (10) to (11);
		\draw [black,-, very thick] (10) to (12);
		\draw [black,-, very thick] (11) to (13);
		\draw [black,-, very thick] (13) to (12);
		\draw [black,-, very thick] (13) to (9);
		\draw [black,-, very thick] (6) to (10);
		\draw [black,-, very thick] (7) to (11);
		\draw [black,-, very thick] (8) to (12);
		\begin{pgfonlayer}{background}
			\highlight{8pt}{turq}{(7.center) to (9.center) to (8.center) to (6.center) to (7.center)}
		\end{pgfonlayer}
\end{tikzpicture}}
\hspace{.5cm}
\subcaptionbox{\label{fig:isometric cycle}}[.3\linewidth]{\begin{tikzpicture}[every circle node/.style={draw, circle ,inner sep=1.25pt},scale=1.25]
\node [circle] (6) at (4.8,0){};
\node [circle] (7) [] at (4.093,0.707){};
\node [circle] (8) [] at (5.507,0.707){};
\node [circle] (9) [] at (4.8,1.414){};
\node [circle] (10) [] at (4.8,-1){};
\node [circle] (11) [] at (4.093,-.293){};
\node [circle] (12) [] at (5.507,-.293){};
\node [circle] (13) at (4.8,0.414){};
\draw [black,-, very thick] (6) to (7);
\draw [black,-, very thick] (6) to (8);
\draw [turq,-, very thick] (7) to (9);
\draw [turq,-, very thick] (9) to (8);
\draw [turq,-, very thick] (10) to (11);
\draw [turq,-, very thick] (10) to (12);
\draw [black,-, very thick] (11) to (13);
\draw [black,-, very thick] (13) to (12);
\draw [black,-, very thick] (13) to (9);
\draw [black,-, very thick] (6) to (10);
\draw [turq,-, very thick] (7) to (11);
\draw [turq,-, very thick] (8) to (12);
\begin{pgfonlayer}{background}
\highlight{8pt}{turq}{(7.center) to (9.center) to (8.center) to (12.center) to (10.center) to (11.center) to (7.center)}
\end{pgfonlayer}
\end{tikzpicture}}
\hspace{.5cm}
\subcaptionbox{\label{fig:cycle}}[.3\linewidth]{
	\begin{tikzpicture}[every circle node/.style={draw, circle ,inner sep=1.25pt},scale=1.25]
		\node [circle] (6) at (4.8,0){};
		\node [circle] (7) [] at (4.093,0.707){};
		\node [circle] (8) [] at (5.507,0.707){};
		\node [circle] (9) [] at (4.8,1.414){};
		\node [circle] (10) [] at (4.8,-1){};
		\node [circle] (11) [] at (4.093,-.293){};
		\node [circle] (12) [] at (5.507,-.293){};
		\node [circle] (13) at (4.8,0.414){};
		\draw [turq,-, very thick] (6) to (7);
		\draw [black,-, very thick] (6) to (8);
		\draw [turq,-, very thick] (7) to (9);
		\draw [turq,-, very thick] (9) to (8);
		\draw [black,-, very thick] (10) to (11);
		\draw [turq,-, very thick] (10) to (12);
		\draw [black,-, very thick] (11) to (13);
		\draw [black,-, very thick] (13) to (12);
		\draw [black,-, very thick] (13) to (9);
		\draw [turq,-, very thick] (6) to (10);
		\draw [black,-, very thick] (7) to (11);
		\draw [turq,-, very thick] (8) to (12);
		\begin{pgfonlayer}{background}
			\highlight{8pt}{turq}{(7.center) to (9.center) to (8.center) to (12.center) to (10.center) to (6.center) to (7.center)}
		\end{pgfonlayer}
\end{tikzpicture}}
\caption{A convex cycle, an isometric cycle that is not convex, and a cycle that is neither isometric nor convex, as described in Example~\ref{ex:isometric vs convex}.}\label{fig:isometric vs convex} 
\end{figure}
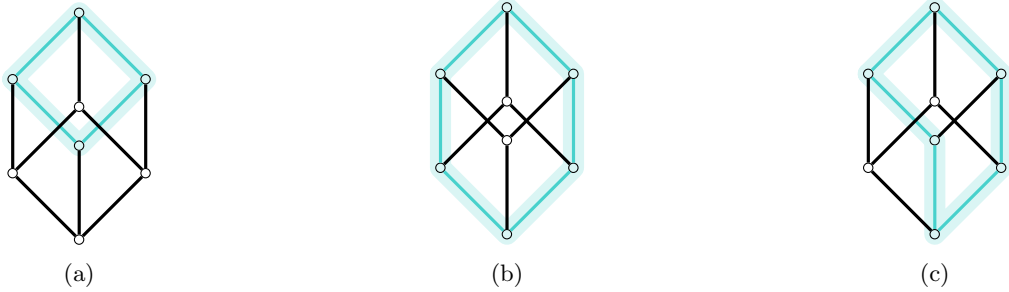

 The \emph{box product} of graphs $G_1$ and $G_2$, denoted $G_1 \square G_2$, is the graph whose vertex set is $V(G_1) \times V(G_2)$ and there is an edge from $(u_1, v_1)$ to $(u_2, v_2)$ provided either:
\begin{itemize}
\item[(a)] $u_1 = u_2$ and $\{v_1,v_2\} \in E(G_2)$, or
\item[(b)] $v_1 = v_2$ and $\{u_1,u_2\} \in E(G_1)$.
\end{itemize}
Note that the box product operation is associative and commutative up to isomorphism.

If $n \in \mathbb{N} \cup \{0\}$, then we define the set of \emph{binary strings} of length $n$ as:
\begin{center}
$\{0,1\}^n := \{a_1a_2 \cdots a_n \mid a_k \in \{0,1\}\}$.
\end{center}
Note that the empty string is the only string of length zero. The \emph{hypercube} of dimension $n \in \mathbb{N} \cup \{0\}$, denoted $Q_n$, is the graph with vertex set $V(Q_n) = \{0, 1\}^n$ and two vertices are adjacent when their corresponding binary strings differ by exactly one digit. Note that for $n,m \in \mathbb{N} \cup \{0\}$, $Q_n \square Q_m \cong Q_{n+m}$. 

A graph $G$ is called a \emph{partial cube} if it can be isometrically embedded in some hypercube $Q_n$. The \emph{isometric dimension} of a partial cube $G$, denoted $\dim_I(G)$, is defined as the minimum dimension of the hypercube into which $G$ can be isometrically embedded. That is,
\[
\dim_I(G):= \min\{n \in \mathbb{N} \cup \{0\} \mid \text{there exists an isometric embedding of} \ G \ \text{into} \ Q_n\}.
\]

\begin{example}\label{ex:partial cubes}
Figures~\ref{fig:partialcube_a} and~\ref{fig:partialcube_b} depict examples of partial cubes together with isometric embeddings into a hypercube. It turns out that the isometric dimensions are 4 and 3, respectively.
\end{example}

\begin{figure}[h!]
\centering
\subcaptionbox{\label{fig:partialcube_a}}[.45\textwidth]{
	\begin{tikzpicture}[every circle node/.style={draw, circle ,inner sep=1.25pt}, scale=0.9]
		\begin{scope}[xshift=.75em]
			\node [circle] (1) [] at (30:1){};
			\node [circle] (2) [] at (90:1){};
			\node [circle] (3) []at (150:1){};
			\node [circle] (4) [] at (210:1){};
			\node [circle] (5) [] at (0,-1){};
			\node [circle] (6) [] at (330:1){};
			\draw [turq,-, very thick] (1) to (2);
			\draw [turq,-, very thick] (1) to (6);
			\draw [turq,-, very thick] (2) to (3);
			\draw [turq,-, very thick] (3) to (4);
			\draw [turq,-, very thick] (4) to (5);
			\draw [turq,-, very thick] (5) to (6);
		\end{scope}
		\begin{scope}[xshift=-1.75em]
			\node [circle] (6) at (4.8,0){};
			\node [circle] (7) [label=left:\small \phantom{101}] at (4.093,0.707){};
			\node [circle] (8) [label=right:\small \phantom{011}] at (5.507,0.707){};
			\node [circle] (9) [label=above:\small \phantom{111}] at (4.8,1.414){};
			\node [circle] (10) [label=below:\small \phantom{000}] at (4.8,-1){};
			\node [circle] (11) [label=left:\small \phantom{100}] at (4.093,-.293){};
			\node [circle] (12) [label=right:\small \phantom{010}] at (5.507,-.293){};
			\node [circle] (13) at (4.8,0.414){};
			\node at (5.75,-.2) {\small \phantom{001}};
			\node at (6.25,0.414+.2) {\small \phantom{110}};
			\draw [black,-, very thick] (6) to (7);
			\draw [black,-, very thick] (6) to (8);
			\draw [turq,-, very thick] (7) to (9);
			\draw [turq,-, very thick] (9) to (8);
			\draw [turq,-, very thick] (10) to (11);
			\draw [turq,-, very thick] (10) to (12);
			\draw [black,-, very thick] (11) to (13);
			\draw [black,-, very thick] (13) to (12);
			\draw [black,-, very thick] (13) to (9);
			\draw [black,-, very thick] (6) to (10);
			\draw [turq,-, very thick] (7) to (11);
			\draw [turq,-, very thick] (8) to (12);
			\begin{pgfonlayer}{background}
				\highlight{8pt}{turq}{(7.center) to (9.center) to (8.center) to (12.center) to (10.center) to (11.center) to (7.center)}
			\end{pgfonlayer}
		\end{scope}
		
		\draw [very thick, right hook->,black] (1.5,0) -- (3,0);
	\end{tikzpicture}
}
\subcaptionbox{\label{fig:partialcube_b}}[.45\textwidth]{
\begin{tikzpicture}[every circle node/.style={draw, circle ,inner sep=1.25pt},scale=.9]
\node [circle] (1) [] at (0,1){};
\node [circle] (2) [] at (0,0.1){};
\node [circle] (3) [] at (-.707,1.707){};
\node [circle] (4) at (.707,1.707){};
\node [circle] (5) [] at (0,2.414){};
\node [circle] (6)  at (1.414,2.414){};
\node [circle] (7) [] at (.707,3.121){};
\draw [turq,-, very thick] (1) to (2);
\draw [turq,-, very thick] (1) to (3);
\draw [turq,-, very thick] (1) to (4);
\draw [turq,-, very thick] (3) to (5);
\draw [turq,-, very thick] (5) to (4);
\draw [turq,-, very thick] (5) to (7);
\draw [turq,-, very thick] (6) to (7);
\draw [turq,-, very thick] (6) to (4);
\begin{scope}[xshift=-1em]
\def\shiftright{6}
\def\shiftup{1.6}
\def\radone{1/(sqrt(2+sqrt(2)))}
\def\radtwo{2.5 * \radone}
\node [circle] (9) at  ({\shiftright +  \radone * cos(0)},  {\shiftup + \radone * sin(0)}){};
\node [circle] (10) at ({\shiftright + \radone * cos(45)},  {\shiftup + \radone * sin(45)}){};
\node [circle] (11) at ({\shiftright + \radone * cos(90)},  {\shiftup + \radone * sin(90)}){};
\node [circle] (12) at ({\shiftright + \radone * cos(135)}, {\shiftup + \radone * sin(135)}){};
\node [circle] (13) at ({\shiftright + \radone * cos(180)}, {\shiftup + \radone * sin(180)}){};
\node [circle] (14) at ({\shiftright + \radone * cos(225)}, {\shiftup + \radone * sin(225)}){};
\node [circle] (15) at ({\shiftright + \radone * cos(270)}, {\shiftup + \radone * sin(270)}){};
\node [circle] (16) at ({\shiftright + \radone * cos(315)}, {\shiftup + \radone * sin(315)}){};
\node [circle] (17) at ({\shiftright + \radtwo * cos(0)},{\shiftup + \radtwo * sin(0)}){};
\node [circle] (18) at ({\shiftright + \radtwo * cos(45)},  {\shiftup + \radtwo * sin(45)}){};
\node [circle] (19) at ({\shiftright + \radtwo * cos(90)},  {\shiftup + \radtwo * sin(90)}){};
\node [circle] (20) at ({\shiftright + \radtwo * cos(135)}, {\shiftup + \radtwo * sin(135)}){};
\node [circle] (21) at ({\shiftright + \radtwo * cos(180)}, {\shiftup + \radtwo * sin(180)}){};
\node [circle] (22) at ({\shiftright + \radtwo * cos(225)}, {\shiftup + \radtwo * sin(225)}){};
\node [circle] (23) at ({\shiftright + \radtwo * cos(270)}, {\shiftup + \radtwo * sin(270)}){};
\node [circle] (24) at ({\shiftright + \radtwo * cos(315)}, {\shiftup + \radtwo * sin(315)}){};
\draw [black,-, very thick] (9) to (24);
\draw [turq,-, very thick] (24) to (23);
\draw [black,-, very thick] (23) to (14);
\draw [black,-, very thick] (14) to (9);

\draw [black,-, very thick] (11) to (18);
\draw [black,-, very thick] (18) to (17);
\draw [turq,-, very thick] (17) to (16);
\draw [black,-, very thick] (16) to (11);

\draw [turq,-, very thick] (13) to (20);
\draw [black,-, very thick] (20) to (19);
\draw [black,-, very thick] (19) to (10);
\draw [black,-, very thick] (10) to (13);

\draw [black,-, very thick] (15) to (22);
\draw [black,-, very thick] (22) to (21);
\draw [black,-, very thick] (21) to (12);
\draw [black,-, very thick] (12) to (15);

\draw [black,-, very thick] (10) to (17);
\draw [turq,-, very thick] (17) to (24);
\draw [black,-, very thick] (24) to (15);
\draw [black,-, very thick] (15) to (10);

\draw [black,-, very thick] (12) to (19);
\draw [black,-, very thick] (19) to (18);
\draw [black,-, very thick] (18) to (9);
\draw [black,-, very thick] (9) to (12);

\draw [black,-, very thick] (14) to (21);
\draw [black,-, very thick] (21) to (20);
\draw [black,-, very thick] (20) to (11);
\draw [black,-, very thick] (11) to (14);

\draw [turq,-, very thick] (16) to (23);
\draw [turq,-, very thick] (23) to (22);
\draw [turq,-, very thick] (22) to (13);
\draw [turq,-, very thick] (13) to (16);

\begin{pgfonlayer}{background}
\highlight{8pt}{turq}{(24.center) to (23.center) to (22.center) to (13.center) to (20.center)} 
\highlight{8pt}{turq}{(23.center) to (16.center) to (13.center)}
\highlight{8pt}{turq}{(16.center) to (17.center) to (24.center)}
\end{pgfonlayer}
\end{scope}
\draw [very thick, right hook->, black] (2.75-1,1.6) -- (4.5-1,1.6) node[midway,above]{};
\end{tikzpicture}}
\caption{Examples of partial cubes.}\label{fig:partialcubes} 

\end{figure}
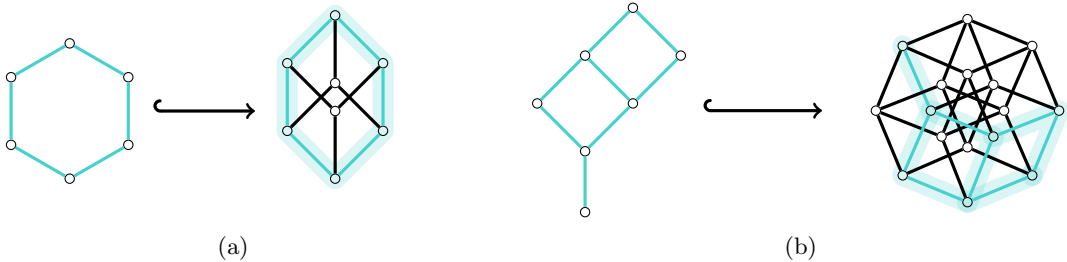

The following result from~\cite{ovchinnikov2008partial} states that the box product of two partial cubes is a partial cube.
\begin{proposition}\label{boxplus}
If $G_1$ and $G_2$ are partial cubes, then $G_1 \square G_2$ is a partial cube with $\dim_I(G_1 \square G_2) = \dim_I(G_1)+\dim_I(G_2)$.
\end{proposition}

The rest of this section mimics the development in~\cite{Mulder1978} and~\cite{ovchinnikov2008partial}. Let $G$ be a graph and let $\{u,v\} \in E(G)$. Define $W_{uv} \subseteq V(G)$ via
\[
W_{uv} := \{w\in V(G) \mid d(w,u) < d(w,v)\}.
\]
That is, $W_{uv}$ is the set of vertices in $G$ that are closer to $u$ than $v$. Both the subgraph $G[W_{uv}]$ and the set $W_{uv}$ are referred to as a \emph{semicube} of $G$. The two semicubes $W_{uv}$ and $W_{vu}$ are referred to as \emph{opposite semicubes}. 

\begin{proposition}\label{prop:semicubeonestepcloser}
Let $G$ be a graph. If $w \in W_{uv}$ for some edge $\{u,v\} \in E(G)$, then $d(w,v) = d(w,u) + 1$. Moreover, $W_{uv} = \{w \in V(G) \mid d(w,v)=d(w,u)+1\}$.
\end{proposition}

That is, if $\{u,v\} \in E(G)$, then all vertices in $W_{vu}$ are exactly one step further from $u$ than $v$ in $G$. Note that some vertices may not be in either semicube, namely the ones equidistant from both $u$ and $v$.

\begin{proposition}\label{prop:bisemis}
A graph $G$ is bipartite if and only if $W_{uv}$ and $W_{vu}$ form a partition of $V(G)$ for any edge $\{u,v\} \in E(G)$. 
\end{proposition}

Using the notion of semicubes, we define the \emph{Djoković--Winkler relation} $\btheta$ on the set of edges of a graph. If $G$ is a graph, we define $\{x,y\} \btheta \{u,v\}$ if and only if $\{u,v\}$ connects a vertex in $W_{xy}$ to a vertex in $W_{yx}$. Note that $\btheta$ is reflexive and symmetric, but not necessarily transitive. 

\begin{example}\label{nottrans}
Figure~\ref{fig:nottrans} provides an example where $\btheta$ is not transitive. The vertices shaded in \textcolor{turq}{teal} are in the semicube $W_{uv}$ while the vertices shaded in \textcolor{magenta}{magenta} are in the semicube $W_{vu}$. If we consider the edges $b_1, b_2$, and $b_3$, we see that $b_1 \btheta b_2$ and $b_1 \btheta b_3$ while $b_2 \not{\btheta} b_3$. 
\end{example}

\begin{figure}[h!]
\centering
\begin{tikzpicture}[every circle node/.style={draw, circle, inner sep=1.25pt},scale=1]
\node [circle] (a) [label=below:$\scriptstyle u$] at (0,0){};
\node [circle] (b) [] at (0,2){};
\node [circle] (c) [label=below:$\scriptstyle v$] at (2,0){};
\node [circle] (d) [] at (2,2){};
\node [circle] (e) [] at (1,1){};

\draw [black,-, very thick] (a) --  (b);
\draw [black,-, very thick] (b) -- (d) node[midway,above]{$\scriptstyle b_3$};
\draw [black,-, very thick] (c) -- (d);
\draw [black,-, very thick] (c) -- (a) node[midway,below]{$\scriptstyle b_1$};
\draw [black,-, very thick] (b) -- (e) node[midway,below]{$\scriptstyle b_2$};
\draw [black,-, very thick] (e) -- (c);

\def\scaleNum{0.3}
\draw[color=turq, fill=turq, opacity=0.3] plot[smooth cycle, tension=.6, very thick] coordinates { (-\scaleNum, \scaleNum+2) (\scaleNum, \scaleNum+2) (\scaleNum, -\scaleNum) (-\scaleNum, -\scaleNum) };

\def\scaleNumm{0.4}
\draw[color=magenta,fill=magenta, opacity=0.3] plot[smooth cycle, tension=.6, very thick] coordinates { (1-\scaleNumm, 1) (2-\scaleNumm, \scaleNumm+2) (2+\scaleNumm, \scaleNumm+2) (2+\scaleNumm, -\scaleNumm) (2-\scaleNumm, -\scaleNumm) };
\end{tikzpicture}
\caption{An example in which $\btheta$ is not transitive as described in Example~\ref{nottrans}.} \label{fig:nottrans}
\end{figure}
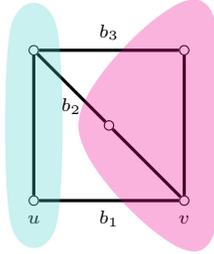

 The following proposition from~\cite{ovchinnikov2008partial} characterizes partial cubes in terms of semicubes and the Djoković--Winkler relation. 

\begin{proposition}\label{prop:TFAE partial cube}
Let $G$ be a graph. The following statements are equivalent:
\begin{itemize}
\item[(i)] $G$ is a partial cube.
\item[(ii)] $G$ is bipartite and all semicubes are convex.
\item[(iii)] $G$ is bipartite and $\btheta$ is an equivalence relation.
\item[(iv)] $G$ is bipartite and for all $\{x,y\}, \{u,v\} \in E(G)$, if $\{x,y\} \btheta \{u,v\}$, then $\{W_{xy}, W_{yx}\} = \{W_{uv}, W_{vu}\}$.
\item[(v)] $G$ is bipartite and for any pair of adjacent vertices of $G$, there is a unique pair of opposite semicubes separating these two vertices.
\end{itemize}
\end{proposition}

If $G$ is a partial cube, then $G$ is bipartite and $\btheta$ is an equivalence relation on $E(G)$. If $G$ is a partial cube and $\{u,v\} \in E(G)$, we denote the equivalence class of $\{u,v\}$ under $\btheta$ as $F_{uv}$. That is,
\[
F_{uv} := \{\{a,b\} \in E(G) \mid \{u,v\}\btheta\{a,b\}\} = \{\{a,b\} \in E(G) \mid a \in W_{uv}, b \in W_{vu}\}.
\]
Notice that $F_{uv}$ is the set of edges joining $W_{uv}$ and $W_{vu}$. We will refer to the edges in $F_{uv}$ as $F$-\emph{edges relative to $W_{uv}$}, or simply \emph{$F$-edges} if the context is clear. Note that while $F_{uv} = F_{vu}$, $W_{uv} \neq W_{vu}$. According to~\cite{Ovchinnikov2008media}, when $G$ is a partial cube, the semicubes $W_{uv}$ and $W_{vu}$ are complementary halfspaces in the metric space $V(G)$. In this case, the set $F_{uv}$ can be regarded as a hyperplane separating these halfspaces.

Figure~\ref{fig:semicubes} gives a rough illustration of the semicubes $W_{uv}$ and $W_{vu}$ for an edge $\{u,v\}$ in a partial cube, where all vertices contained within the \textcolor{turq}{teal} box are closer to $u$ than $v$ and all the vertices contained in the \textcolor{magenta}{magenta} box are closer to $v$ than $u$. The edges in black are the edges in $F_{uv}$. Notice that the \textcolor{magenta}{magenta} box is smaller than the \textcolor{turq}{teal} box, illustrating that opposite semicubes need not have the same cardinality.

The next proposition from~\cite{ovchinnikov2008partial} provides a method for computing the isometric dimension of a partial cube.

\begin{proposition}\label{prop:isometric dimension}
If $G$ is a partial cube, then $\dim_I(G)$ is equal to the number of equivalence classes induced by the Djoković--Winkler relation $\btheta$.
\end{proposition}

\begin{figure}[h!]
\centering
\begin{tikzpicture}[every circle node/.style={draw, circle ,inner sep=1.25pt}, scale=1.25]
\draw[line width=1.5pt, draw=turq, fill=turq, opacity=0.3] (0.25,-0.25) rectangle (2.5,3.25);
\node [circle] (1) [label=left:$\scriptstyle u$] at (1.75,0.75){}; 
\node [circle] (5) [] at (1.75,1){};
\node [circle] (2) [] at (1.75, 2){};
\node (x) at (1.5,3.5) {$W_{uv}$};
\draw[line width=1.5pt, draw=magenta, fill=magenta, opacity=0.3] (3,.5) rectangle (4.5,2.25);
\node [circle] (3) [label=right:$\scriptstyle v$] at (3.75,0.75){}; 
\node [circle] (6) [] at (3.75,1){};
\node [circle] (4) [] at (3.75, 2){};
\draw [black,-, very thick] (1) to (3);
\draw [black,-, very thick] (2) to (4);
\draw [black,-, very thick] (5) to (6);
\node (y) at (3.75,2.5) {$W_{vu}$};
\node (z) at (2.75,1.5) {$\vdots$};
\end{tikzpicture}

\caption{A rough illustration of two semicubes for a partial cube together with the corresponding class of $F$-edges.}
\label{fig:semicubes}
\end{figure}
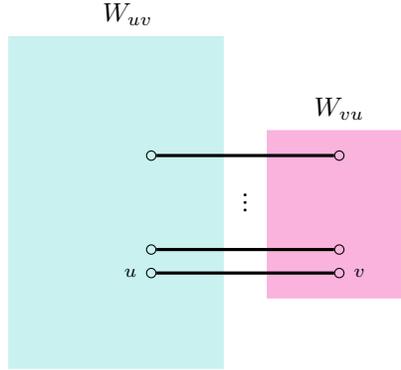

\begin{example}\label{ex:semicubes}
Consider the partial cube in Figure~\ref{fig:semi cubes}. The semicube $W_{u_1 v_1}$ is highlighted in \textcolor{turq}{teal} while the opposite semicube $W_{v_1 u_1}$ is highlighted in \textcolor{magenta}{magenta}. The corresponding $F$-edges are colored black. In Figure~\ref{fig:F-edges}, we utilized five colors to indicate each of the five equivalence classes of $F$-edges. It follows that the isometric dimension of this graph is 5 according to Proposition~\ref{prop:isometric dimension}.
\end{example}

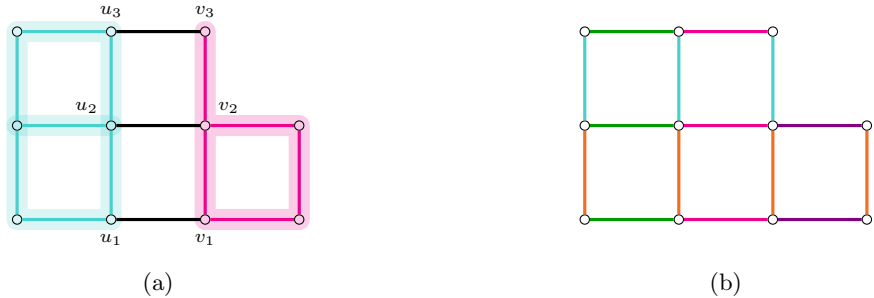
\begin{figure}[h!]
\centering
\subcaptionbox{\label{fig:semi cubes}}[.45\textwidth]{
\begin{tikzpicture}[every circle node/.style={draw, circle, inner sep=1.25pt},scale=1.25]
\node [circle] (a) [label=below:$\scriptstyle u_1$] at (0,0){};
\node [circle] (b) [label=above left:$\scriptstyle u_2$] at (0,1){};
\node [circle] (c) [label=below:$\scriptstyle v_1$] at (1,0){};
\node [circle] (d) [label=above right:$\scriptstyle v_2$] at (1,1){};

\node [circle] (e) [label=above:$\scriptstyle u_3$] at (0,2){};
\node [circle] (f) [label=above:$\scriptstyle v_3$] at (1,2){};
\node [circle] (g) [] at (2,1){};
\node [circle] (h) [] at (2,0){};

\node [circle] (i) [] at (-1,0){};
\node [circle] (j) [] at (-1,1){};
\node [circle] (k) [] at (-1,2){};


\draw [turq,-, very thick] (a) to (b);
\draw [black,-, very thick] (b) to (d);
\draw [magenta,-, very thick] (c) to (d);
\draw [black,-, very thick] (c) to (a);
\draw [turq,-, very thick] (b) to (e);
\draw [black,-, very thick] (e) to (f);
\draw [magenta,-, very thick] (f) to (d);
\draw [magenta,-, very thick] (d) to (g);
\draw [magenta,-, very thick] (g) to (h);
\draw [magenta,-, very thick] (h) to (c);
\draw [turq, -, very thick] (i) to (j);
\draw [turq, -, very thick] (j) to (k);
\draw [turq, -, very thick] (i) to (a);
\draw [turq, -, very thick] (j) to (b);
\draw [turq, -, very thick] (k) to (e);

\begin{pgfonlayer}{background}
\highlight{8pt}{turq}{(a.center) to (b.center) to (e.center) to (k.center) to (j.center) to (i.center) to (a.center)}
\highlight{8pt}{turq}{(j.center) to (b.center)}
\highlight{8pt}{magenta}{(f.center) to (d.center) to (g.center) to (h.center) to (c.center) to (d.center)}
\end{pgfonlayer}
\end{tikzpicture}
}
\subcaptionbox{\label{fig:F-edges}}[.45\textwidth]{
\begin{tikzpicture}[every circle node/.style={draw, circle, inner sep=1.25pt},scale=1.25]
\node [circle] (a) [label=below:$\phantom{\scriptstyle u_1}$] at (0,0){};
\node [circle] (b) [] at (0,1){};
\node [circle] (c) [] at (1,0){};
\node [circle] (d) [] at (1,1){};

\node [circle] (e) [] at (0,2){};
\node [circle] (f) [] at (1,2){};
\node [circle] (g) [] at (2,1){};
\node [circle] (h) [] at (2,0){};

\node [circle] (i) [] at (-1,0){};
\node [circle] (j) [] at (-1,1){};
\node [circle] (k) [] at (-1,2){};

\draw [nectarine,-, very thick] (a) to (b);
\draw [magenta,-, very thick] (b) to (d);
\draw [nectarine,-, very thick] (c) to (d);
\draw [magenta,-, very thick] (c) to (a);
\draw [turq,-, very thick] (b) to (e);
\draw [magenta,-, very thick] (e) to (f);
\draw [turq,-, very thick] (f) to (d);
\draw [violet,-, very thick] (d) to (g);
\draw [nectarine,-, very thick] (g) to (h);
\draw [violet,-, very thick] (h) to (c);
\draw [nectarine, -, very thick] (i) to (j);
\draw [turq, -, very thick] (j) to (k);
\draw [ggreen, -, very thick] (i) to (a);
\draw [ggreen, -, very thick] (j) to (b);
\draw [ggreen, -, very thick] (k) to (e);
\end{tikzpicture}
}
\caption{Example of semicubes for the partial cube from Example~\ref{ex:semicubes} together with the corresponding equivalence classes of $F$-edges.} \label{fig:semicubeequivclass}
\end{figure}

%
%

If $G$ is a graph and $\{u,v\}\in E(G)$, then we define the following sets:
\begin{align*}
U_{uv} &:=\{w\in W_{uv}\mid w\text{ is adjacent to a vertex in }W_{vu}\},\\
U_{vu} &:=\{w\in W_{vu}\mid w\text{ is adjacent to a vertex in }W_{uv}\}.
\end{align*}

The next result appears in~\cite{ovchinnikov2008partial}.

\begin{proposition}\label{prop:iso for U sets}
If $G$ is a partial cube with $\{u,v\} \in E(G)$, then the set $F_{uv}$ is a matching and induces an isomorphism between the induced subgraphs $G[U_{uv}]$ and $G[U_{vu}]$.
\end{proposition}

\begin{example}
As an example of the previous proposition, the black edges in Figure~\ref{fig:semi cubes} induce an isomorphism between the subgraphs induced by $U_{u_1v_1}=\{u_1,u_2,u_3\}$ and $U_{v_1u_1}=\{v_1,v_2,v_3\}$.
\end{example}

The \emph{interval} between vertices $u$ and $v$ of a graph $G$ is the set
\[
I(u,v):=\{w\in V(G) \mid d(u,v) = d(u,w) + d(w,v)\}.
\]
That is, $I(u,v)$ is the collection of vertices that lie on some geodesic between $u$ and $v$. A graph $G$ is \emph{median} if 
\[
|I(u,v) \cap I(u,w) \cap I(v,w)| = 1
\]
for all $u,v,w \in V(G)$. In other words, $G$ is median if there is a unique vertex $x$ that simultaneously lies on a geodesic between $u$ and $v$, a geodesic between $u$ and $w$, and a geodesic between $v$ and $w$ for all triples $u,v,w$. For a median graph $G$, define $\med(u,v,w)$ to be the unique vertex in $I(u, v) \cap I(u, w) \cap I(v, w)$. 

\begin{example}\label{ex:median}
The shading in Figures~\ref{fig:median1} and~\ref{fig:notmedian1} depicts $I(u,v)$ in \textcolor{red}{red}, $I(v,w)$ in \textcolor{blue}{blue}, and $I(u,w)$ in \textcolor{ggreen}{green}. In Figure~\ref{fig:median1}, we see that all three colors overlap at the vertex $x$, illustrating that $| I(u,v) \cap I(u,w) \cap I(v,w) | = 1$. It turns out that the same is true for any three vertices in this graph. Thus, the graph given in Figure~\ref{fig:median1} is median. On the other hand, in Figure~\ref{fig:notmedian1}, we see that there is no vertex common to all of these intervals, and so the graph in Figure~\ref{fig:notmedian1} is not median.
\end{example}

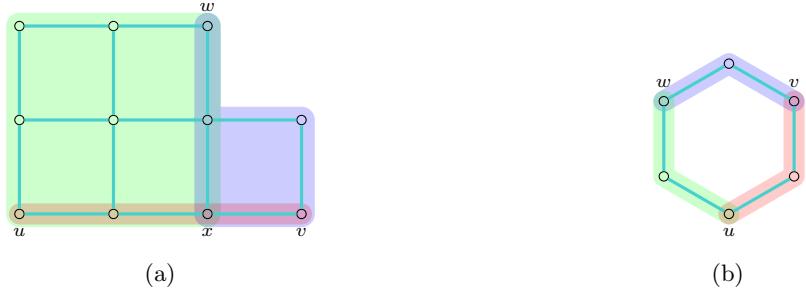
\begin{figure}[h!]
\centering
\subcaptionbox{\label{fig:median1}}[.45\linewidth]{
\begin{tikzpicture}[every circle node/.style={draw, circle, inner sep=1.25pt},scale=1.25]
\node [circle] (a) [] at (0,0){};
\node [circle] (b) [] at (0,1){};
\node [circle] (c) [label=below:$\scriptstyle x$] at (1,0){};
\node [circle] (d) [] at (1,1){};

\node [circle] (e) [] at (0,2){};
\node [circle] (f) [label=above:$\scriptstyle w$] at (1,2){};
\node [circle] (g) [] at (2,1){};
\node [circle] (h) [label=below:$\scriptstyle v$] at (2,0){};

\node [circle] (i) [label=below:$\scriptstyle u$] at (-1,0){};
\node [circle] (j) [] at (-1,1){};
\node [circle] (k) [] at (-1,2){};
\draw [turq,-, very thick] (a) to (b);
\draw [turq,-, very thick] (b) to (d);
\draw [turq,-, very thick] (c) to (d);
\draw [turq,-, very thick] (c) to (a);
\draw [turq,-, very thick] (b) to (e);
\draw [turq,-, very thick] (e) to (f);
\draw [turq,-, very thick] (f) to (d);
\draw [turq,-, very thick] (d) to (g);
\draw [turq,-, very thick] (g) to (h);
\draw [turq,-, very thick] (h) to (c);
\draw [turq,-, very thick] (i) to (j);
\draw [turq,-, very thick] (j) to (k);
\draw [turq,-, very thick] (i) to (a);
\draw [turq,-, very thick] (j) to (b);
\draw [turq,-, very thick] (k) to (e);
\begin{pgfonlayer}{background}
\fill[green,opacity=0.2] \convexpath{a,i,j,k,e,f,d,c}{4pt};
\fill[blue,opacity=0.2] \convexpath{g,h,c,d,f,d}{4pt};
\highlight{8pt}{red,opacity=.2}{(i.center) to (a.center) to (h.center)};
\end{pgfonlayer}
\end{tikzpicture}}
\subcaptionbox{\label{fig:notmedian1}}[.45\linewidth]{
\begin{tikzpicture}[every circle node/.style={draw, circle ,inner sep=1.25pt}]
\node [circle] (1) [label=$\scriptstyle v$] at (30:1){};
\node [circle] (2) [] at (90:1){};
\node [circle] (3) [label=$\scriptstyle w$]at (150:1){};
\node [circle] (4) [] at (210:1){};
\node [circle] (5) [label=below:$\scriptstyle u$] at (0,-1){};
\node [circle] (6) [] at (330:1){};
\draw [turq,-, very thick] (1) to (2);
\draw [turq,-, very thick] (1) to (6);
\draw [turq,-, very thick] (2) to (3);
\draw [turq,-, very thick] (3) to (4);
\draw [turq,-, very thick] (4) to (5);
\draw [turq,-, very thick] (5) to (6);
\begin{pgfonlayer}{background}
\highlight{8pt}{blue,opacity=.2}{(3.center) to (2.center) to (1.center)}
\highlight{8pt}{green,opacity=.2}{(5.center) to (4.center) to (3.center)}
\highlight{8pt}{red,opacity=.2}{(5.center) to (6.center) to (1.center)}
\end{pgfonlayer}
\end{tikzpicture}}
\caption{Examples of a median graph and non-median graph from Example~\ref{ex:median}.} \label{fig:median}
\end{figure} 

The next proposition from~\cite{ovchinnikov2011} connects partial cubes and median graphs.  

\begin{proposition}\label{prop:median implies partial cube}
If a graph $G$ is median, then $G$ is a partial cube.
\end{proposition}

\begin{example}
As seen in Example~\ref{ex:partial cubes}, the cycle graph with six vertices from Figure~\ref{fig:partialcube_a} can be isometrically embedded into a hypercube, and is therefore a partial cube. However, as shown in Example~\ref{ex:median}, this graph is not median, so the converse of the previous proposition does not hold.
\end{example}

The following result is commonly known and states that, like the collection of partial cubes, the collection of median graphs is closed under the box product operation.

\begin{proposition}\label{medprodmed}
If graphs $G_1$ and $G_2$ are median, then $G_1 \square G_2$ is also median.
\end{proposition}

To conclude this section on graphs, we discuss peripheral expansions and their relationship to median graphs. Given a graph $G$ and a convex set $U \subseteq V(G)$, we define the \emph{peripheral expansion of $G$ along $U$}, denoted $P(G,U)$, as follows:
\begin{itemize}
\item Start with the graph $G$;
\item Form the disjoint union of $G$ and an isomorphic copy of $G[U]$, denoted $G'_U$, where each $u \in U$ corresponds to $u' \in U' := V(G'_U)$; 
\item For each $u \in U$, join $u$ and $u'$ with an edge.
\end{itemize}
Note that peripheral expansions are a special case of convex expansions as described in~\cite{Mulder1978}.

The illustration given in Figure~\ref{fig:expansion} shows a rough depiction of the process described above. The vertices in $G[U]$ mirror the vertices in the \textcolor{magenta}{magenta} $G'_U$, and each pair of vertices $u$ and $u'$ are connected by a black edge. The rectangles $G[U]$ and $G'_U$ are drawn the same size to indicate the isomorphism between the two graphs: $\{u,v\} \in E(G[U])$ if and only if $\{u',v'\} \in E(G'_U)$. When $G$ is a partial cube, the \textcolor{turq}{teal} $G$ and \textcolor{magenta}{magenta} $G_U'$ are opposite semicubes and the black edges joining each $u$ and $u'$ pair are all part of the same $F$-class of edges.

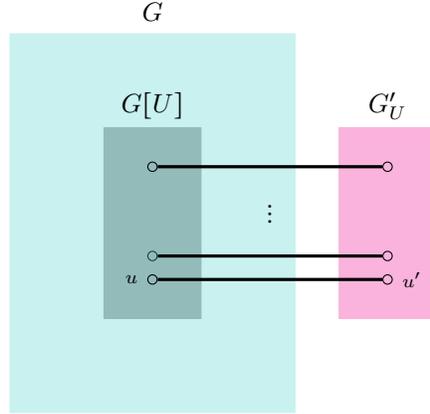
\begin{figure}[h!]
\centering
 \begin{tikzpicture}[every circle node/.style={draw, circle ,inner sep=1.25pt}, scale=1.25]
\node [circle] (6) [] at (.5,.65){};
\draw[line width=1.5pt, draw=black, fill=black, opacity=0.3] (0,0) rectangle (1,2);
\draw[line width=1.5pt, draw=magenta, fill=magenta, opacity=0.3] (2.5,0) rectangle (3.5,2);
\draw[line width=1.5pt, draw=turq, fill=turq, opacity=0.3] (-1,-1) rectangle (2,3);
\node [circle] (1) [label=left:$\scriptstyle u$] at (.5,.4){};
\node [circle] (7) [] at (3,.65){};
\node [circle] (4) [] at (.5,1.6){};
\node [circle] (5) [label=right:$\scriptstyle u'$] at (3,.4){}; 
\node [circle] (8) [] at (3,1.6){};
\node (x) at (.5,3.25) {$G$};
\node (x) at (.5,2.25) {$G[U]$};
\draw [black,-, very thick] (1) to (5);
\draw [black,-, very thick] (6) to (7);
\draw [black,-, very thick] (4) to (8);
\node (x) at (3,2.25) {$G_U'$};
\node (z) at (1.75,1.1) {$\vdots$};
\end{tikzpicture}

\caption{A rough illustration of the peripheral expansion process.}
\label{fig:expansion}
\end{figure}

\begin{example}
Figure~\ref{fig:convexexp} illustrates a sequence of peripheral expansions starting from a single vertex. The grey highlighted portion of each subfigure shows which subgraph is playing the role of $G[U]$. Each subsequent graph represents the graph obtained when the peripheral expansion is performed along the convex grey portion.
\end{example}

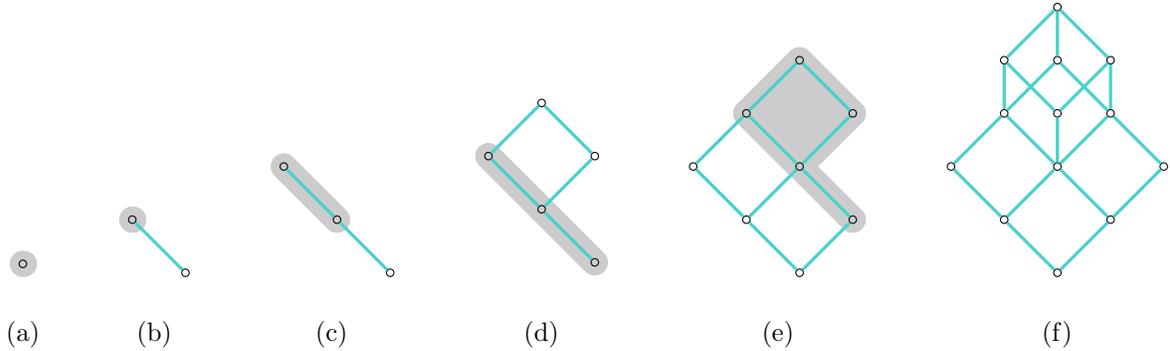
\begin{figure}[!h]
	\begin{center}
	\begin{tabular}{ccccccccccc}
		\begin{tikzpicture}[every circle node/.style={draw, circle,inner sep=1pt}]
		\node [circle] (1) at (.707,.293){};
		\begin{pgfonlayer}{background}
		\highlight{10pt}{black}{(1.center) to (1.center)};
		\end{pgfonlayer}
		\end{tikzpicture}
	& \hspace{1em} & 
		\begin{tikzpicture}[every circle node/.style={draw, circle,fill=none,inner sep=1pt}]
			\node [circle] (1) at (0,1){};
			\node [circle] (2) at (.707,.293){};
			\draw [turq,-, very thick] (1) to (2);
							\begin{pgfonlayer}{background}
								\highlight{10pt}{black}{(1.center) to (1.center)};
							\end{pgfonlayer}
		\end{tikzpicture}
		& \hspace{1em} &
		\begin{tikzpicture}[every circle node/.style={draw, circle,fill=none, inner sep=1pt}]
			\node [circle] (1) at (0,1){};
			\node [circle] (2) at (.707,.293){};
			\node [circle] (3) at (-.707,1.707){};
			\draw [turq,-, very thick] (1) to (2);
			\draw [turq,-, very thick] (1) to (3);
							\begin{pgfonlayer}{background}
								\fill[black,opacity=0.2] \convexpath{1,3}{5pt};
							\end{pgfonlayer}
		\end{tikzpicture}
		&  \hspace{1em}  &
		\begin{tikzpicture}[every circle node/.style={draw, circle,fill=none, inner sep=1pt}]
			\node [circle] (1) at (0,1){};
			\node [circle] (2) at (.707,.293){};
			\node [circle] (3) at (-.707,1.707){};
			\node [circle] (4) at (.707,1.707){};
			\node [circle] (5) at (0,2.414){};
			\draw [turq,-, very thick] (1) to (2);
			\draw [turq,-, very thick] (1) to (3);
			\draw [turq,-, very thick] (1) to (4);
			\draw [turq,-, very thick] (3) to (5);
			\draw [turq,-, very thick] (4) to (5);
							\begin{pgfonlayer}{background}
								\fill[black,opacity=0.2] \convexpath{1,2,3}{5pt};
							\end{pgfonlayer}
		\end{tikzpicture}
		&  \hspace{1em}  &
		\begin{tikzpicture}[every circle node/.style={draw, circle,fill=none, inner sep=1pt}]
			\node [circle] (1) at (0,1){};
			\node [circle] (2) at (.707,.293){};
			\node [circle] (3) at (-.707,1.707){};
			\node [circle] (4) at (.707,1.707){};
			\node [circle] (5) at (0,2.414){};
			\node [circle] (6) at (-.707,.293){};
			\node [circle] (7) at (0,-0.414){};
			\node [circle] (8) at (-1.414,1){};
			\draw [rotate = 90, turq,-, very thick] (1) to (2);
			\draw [turq,-, very thick] (1) to (3);
			\draw [turq,-, very thick] (1) to (4);
			\draw [turq,-, very thick] (1) to (6);
			\draw [turq,-, very thick] (8) to (3);
			\draw [turq,-, very thick] (8) to (6);
			\draw [turq,-, very thick] (3) to (5);
			\draw [turq,-, very thick] (6) to (7);
			\draw [turq,-, very thick] (5) to (4);
			\draw [turq,-, very thick] (7) to (2);
							\begin{pgfonlayer}{background}
								\fill[black,opacity=0.2] \convexpath{3,5,4,1,3,2}{5pt};
							\end{pgfonlayer}
		\end{tikzpicture}
		&  \hspace{1em}  &
		\begin{tikzpicture}[every circle node/.style={draw, circle,fill=white,inner sep=1pt}]
			\node [circle] (1) at (0,1){};
			\node [circle] (2) at (.707,.293){};
			\node [circle] (3) at (-.707,1.707){};
			\node [circle] (4) at (.707,1.707){};
			\node [circle] (5) at (0,2.414){};
			\node [circle] (6) at (-.707,.293){};
			\node [circle] (7) at (0,-0.414){};
			\node [circle] (8) at (-1.414,1){};
			\draw [rotate = 90, turq,-, very thick] (1) to (2);
			\draw [turq,-, very thick] (1) to (3);
			\draw [turq,-, very thick] (1) to (4);
			\draw [turq,-, very thick] (1) to (6);
			\draw [turq,-, very thick] (8) to (3);
			\draw [turq,-, very thick] (8) to (6);
			\draw [turq,-, very thick] (3) to (5);
			\draw [turq,-, very thick] (6) to (7);
			\draw [turq,-, very thick] (5) to (4);
			\draw [turq,-, very thick] (7) to (2);
			\node [circle] (9) at (1.414,1){};
			\node [circle] (10) at (0,1.707){};
			\node [circle] (11) at (0,3.121){};
			\node [circle] (13) at (.707,2.414){};
			\node [circle] (14) at (-.707,2.414){};
			\draw [turq,-, very thick] (1) to (10);
			\draw [turq,-, very thick] (10) to (13);
			\draw [turq,-, very thick] (10) to (14);
			\draw [turq,-, very thick] (14) to (11);
			\draw [turq,-, very thick] (13) to (11);
			\draw [turq,-, very thick] (13) to (4);
			\draw [turq,-, very thick] (4) to (9);
			\draw [turq,-, very thick] (9) to (2);
			\draw [turq,-, very thick] (5) to (11);
			\draw [turq,-, very thick] (14) to (3);
		\end{tikzpicture}\\
	& & & & & & & & & &\\
	(a) & & (b) & & (c) & & (d) & & (e) & & (f)
	\end{tabular}
\end{center}
\caption{A sequence of peripheral expansions starting from a single vertex.}
\label{fig:convexexp}
\end{figure}

The following result appears in~\cite{Mulder2010} and is a special case of a result often referred to as Mulder's Theorem~\cite{Mulder1978}. It states that a median graph can always be obtained through a sequence of peripheral expansions that begin from a single vertex.

\begin{proposition}\label{prop:mulder}
A graph $G$ is median if and only if it can be obtained from a single vertex by a sequence of peripheral expansions.
\end{proposition}

\begin{example}
Proposition~\ref{prop:mulder} implies that each graph in Figure~\ref{fig:convexexp} is median.
\end{example}

\section{Coxeter systems and braid graphs}\label{sec:Coxeter systems}

A \emph{Coxeter matrix} is an $n\times n$ symmetric matrix $M=(m_{ij})$ with entries $m_{ij}\in\{1,2,3,\ldots,\infty\}$ such that $m_{ii}=1$ for all $1\leq i\leq n$ and $m_{ij}\geq 2$ for $i\neq j$. A \emph{Coxeter system} is a pair $(W,S)$ consisting of a finite set $S=\{s_1,s_2,\ldots,s_n\}$ and a group $W$, called a \emph{Coxeter group}, with presentation
\[
W = \langle s_1,s_2,\ldots,s_n \mid (s_is_j)^{m(s_i,s_j)}=e\rangle,
\] 
where $m(s_i,s_j):=m_{ij}$ for some $n\times n$ Coxeter matrix $M=(m_{ij})$.  For $s,t\in S$, the condition $m(s,t)=\infty$ means that there is no relation imposed between $s$ and $t$. It turns out that the elements of $S$ are distinct as group elements and $m(s,t)$ is the order of $st$~\cite{Humphreys1990}. Since elements of $S$ have order two, the relation $(st)^{m(s,t)} = e$ can be written as
\[
\underbrace{sts \cdots}_{m(s,t)} = \underbrace{tst \cdots}_{m(s,t)}
\]
with $m(s,t)\geq 2$ letters on each side.  When $m(s,t)=2$, $st=ts$ is called a \emph{commutation relation} and when $m(s,t) \geq 3$, the corresponding relation is called a \emph{braid relation}. The replacement
\[
\underbrace{sts\cdots}_{m(s,t)} \longmapsto  \underbrace{tst\cdots}_{m(s,t)}
\]
is called a \emph{commutation move} if $m(s,t)=2$ and a \emph{braid move} if $m(s,t) \geq 3$.

We can visually encode the information given in a Coxeter system into a \emph{Coxeter graph}, $\Gamma$, having vertex set $S$ and edges $\{s,t\}$ for each $m(s,t)\geq 3$. Moreover, each edge is labeled with the corresponding $m(s,t)$, although typically the labels of $3$ are omitted because they are the most common. We say that $(W,S)$, or just $W$, is type $\Gamma$, and we may denote the Coxeter group as $W(\Gamma)$ and the generating set as $S(\Gamma)$ for emphasis.

In this paper, our focus will be on a special class of Coxeter systems. A Coxeter system is \emph{simply laced} if $m(s,t) \leq 3$ for all $s,t \in S$. That is, a Coxeter system is said to be simply laced if any pair of distinct generators either commute or satisfy a braid relation of length three. If a Coxeter graph $\Gamma$ contains no three-cycles, we say that the corresponding Coxeter system $(W,S)$ is \emph{triangle free}. A Coxeter system that is both simply laced and triangle free is said to be of \emph{type $\Lambda$}.
\begin{example}
The Coxeter graphs given in Figure~\ref{fig:labeledgraphs} correspond to four common simply-laced Coxeter systems. The Coxeter group $W(A_n)$ is isomorphic to the symmetric group $S_{n+1}$ under the mapping that sends $s_i$ to the adjacent transposition $(i,i+1)$. The Coxeter group $W(D_n)$ is isomorphic to the index two subgroup of the group of signed permutations on $n$ letters having an even number of sign changes. The Coxeter systems of types $\widetilde{A}_n$ and $\widetilde{D}_n$ depicted in Figures~\ref{fig:CoxgraphaffA_n} and~\ref{fig:CoxgraphaffD_4}, respectively, turn out to yield infinite Coxeter groups. All of these Coxeter systems are of type $\Lambda$ except type $\widetilde{A}_2$ since its Coxeter graph is a 3-cycle.
\end{example}

\begin{figure}[h]
\centering
\subcaptionbox{\label{fig:labeledA}$A_{n}$}[.45\linewidth]{
\begin{tikzpicture}[every circle node/.style={draw, circle, inner sep=1.25pt}]
\node [circle] (1) [label=below:$s_1$] at (0,0){};
\node [circle] (2) [label=below:$s_2$] at (1,0){};
\node [circle] (3) [label=below:$s_3$] at (2,0){};
\node at (2.5,0) {$\dots$};
\node [circle] (4) [label=below:$s_{n-1}$] at (3,0){};
\node [circle] (5) [label=below:$s_n$] at (4,0){};
\draw [-, very thick] (1) to (2);
\draw [-, very thick] (2) to (3);
\draw [-, very thick] (4) to (5);
\end{tikzpicture}
}
\subcaptionbox{\label{fig:labeledB}$D_n$}[.45\linewidth]{
\begin{tikzpicture}[every circle node/.style={draw, circle, inner sep=1.25pt}]
\node [circle] (1) [label=below:$s_2$] at (0.25,-0.5){};
\node [circle] (2) [label=below:$s_3$] at (1,0){};
\node [circle] (3) [label=below:$s_4$] at (2,0){};
\node at (2.5,0) {$\dots$};
\node [circle] (4) [label=below:$s_{n-1}$] at (3,0){};
\node [circle] (5) [label=below:$s_{n}$] at (4,0){};
\node [circle] (6) [label=above:$s_1$] at (.25,.5){};
\node [label=below:\phantom{$s_{n+1}$}](8) at (4.75,0){};
\draw [-, very thick] (1) to (2);
\draw [-, very thick] (2) to (3);
\draw [-, very thick] (4) to (5);
\draw [-, very thick] (2) to (6);
\end{tikzpicture}
}
\subcaptionbox{\label{fig:CoxgraphaffA_n}$\widetilde{A}_n$}[.45\linewidth]{
\begin{tikzpicture}[every circle node/.style={draw, circle, semithick, inner sep=1.25pt}]
\node [circle] (1) [label=below:$s_1$] at (0,0){};
\node [circle] (2) [label=below:$s_2$] at (1,0){};
\node [circle] (3) [label=below:$s_3$] at (2,0){};
\node at (2.5,0) {$\dots$};
\node [circle] (4) [label=below:$s_{n-1}$] at (3,0){};
\node [circle] (5) [label=below:$s_n$] at (4,0){};
\node [circle] (6) [label={[label distance=2mm]right:$s_{n+1}$}] at (2,1){};
\draw [-, very thick] (1) to (2);
\draw [-, very thick] (2) to (3);
\draw [-, very thick] (4) to (5);
\draw [-, very thick] (1) to (6);
\draw [-, very thick] (5) to (6);
\end{tikzpicture}
}
\subcaptionbox{\label{fig:CoxgraphaffD_4}$\widetilde{D}_n$}[.45\linewidth]{
\begin{tikzpicture}[every circle node/.style={draw, circle, inner sep=1.25pt}]
\node [circle] (1) [label=below:$s_2$] at (0.25,-0.5){};
\node [circle] (2) [label=below:$s_3$] at (1,0){};
\node [circle] (3) [label=below:$s_4$] at (2,0){};
\node at (2.5,0) {$\dots$};
\node [circle] (4) [label=below:$s_{n-2}$] at (3,0){};
\node [circle] (5) [label=below:$s_{n-1}$] at (4,0){};
\node [circle] (6) [label=above:$s_1$] at (.25,.5){};
\node [circle] (7) [label=above:$s_{n}$] at (4.75,.5){};
\node [circle] (8) [label=below:$s_{n+1}$] at (4.75,-.45){};
\draw [-, very thick] (1) to (2);
\draw [-, very thick] (2) to (3);
\draw [-, very thick] (4) to (5);
\draw [-, very thick] (2) to (6);
\draw [-, very thick] (5) to (7);
\draw [-, very thick] (5) to (8);
\end{tikzpicture}
}
\caption{Examples of common simply-laced Coxeter graphs.}\label{fig:labeledgraphs}
\end{figure}
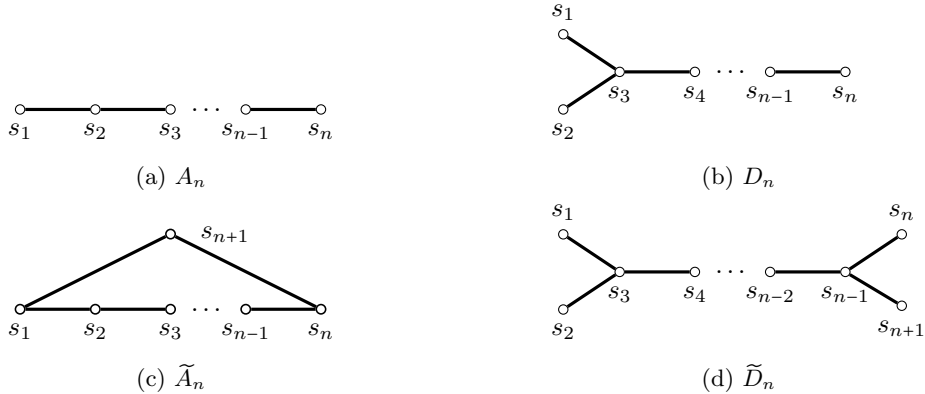

Consider a Coxeter system $(W,S)$. Define $S^*$ to be the free monoid on $S$. We call $\ralpha= s_{x_1}s_{x_2} \cdots s_{x_m} \in S^*$ a \emph{word} while a \emph{factor} of $\ralpha$ is a word of the form $s_{x_i}s_{x_{i+1}} \cdots s_{x_{j-1}}s_{x_j}$ for $1 \leq i \leq j \leq m$. We will write $\rbeta \leq \ralpha$ if $\rbeta$ is a factor of $\ralpha$. The relation $\leq$ makes the set of all factors of $\ralpha$ into a partially ordered set. Now, let $w \in W$. If $\ralpha = s_{x_1}s_{x_2} \cdots s_{x_m} \in S^*$ is equal to $w$ when considered as an element of the group $W$, we say that $\ralpha$ is an \emph{expression} for $w$. If $m$ is minimal among all possible expressions for $w$, we say that $\ralpha$ is a \emph{reduced expression} for $w$. We define the \emph{length} of $w$, denoted $\ell(w)$, to be the number of letters in a reduced expression. We will also say that any reduced expression for $w$ has length $\ell(w)$. Note that any factor of a reduced expression is also reduced. We denote the set of all reduced expressions for a group element $w \in W$ by $\mathcal{R}(w)$. For brevity, if we are considering a particular labeling of a Coxeter graph, we will often replace $s_i$ with $i$.

The following result, called Matsumoto's Theorem~\cite[Theorem~1.2.2]{Geck2000}, characterizes the relationship among reduced expressions for a given group element.

\begin{proposition}[Matsumoto's Theorem]{\label{prop:matsumoto}}
In a Coxeter system $(W,S)$, any two reduced expressions for the same group element differ by a sequence of commutation and braid moves.
\end{proposition}

In light of Matsumoto's Theorem, we can define a graph on the set of reduced expressions of a given element in a Coxeter group. For a Coxeter system $(W,S)$ and $w \in W$, the \emph{Matsumoto graph} $\mathcal{G}(w)$ is defined to be the graph whose vertex set is $\mathcal{R}(w)$ and two vertices $\ralpha$ and $\rbeta$ are connected by an edge if and only if $\ralpha$ and $\rbeta$ are related via a single commutation or braid move. Temporarily, we will color an edge \textcolor{nectarine}{orange} if it corresponds to a commutation move and we will color an edge \textcolor{turq}{teal} if it corresponds to a braid move. Matsumoto's Theorem implies that $\mathcal{G}(w)$ is connected. In~\cite{Bergeron2015}, Bergeron, Ceballos, and Labbé proved that every cycle in a Matsumoto graph for finite Coxeter groups is of even length. This result was extended to arbitrary Coxeter systems in~\cite{Grinberg2017}. As a result of this fact, we get the following proposition.

\begin{proposition}\label{prop:Masumoto bipartite}
If $(W,S)$ is a Coxeter system and $w \in W$, then $\mathcal{G}(w)$ is bipartite. 
\end{proposition}

\begin{example}\label{ex:matsumotoD4}
Consider the expression $\ralpha=1321434$ for some $w$ in the Coxeter system of type $D_4$. It turns out that $\ralpha$ is reduced, so that $\ell(w)=7$. Moreover, there are 15 reduced expressions in $\mathcal{R}(w)$ and the corresponding Matsumoto graph is given in Figure~\ref{fig:Matsumoto graph}. The edges of $\mathcal{G}(w)$ show how pairs of reduced expressions are related via commutation or braid moves. 
\end{example}

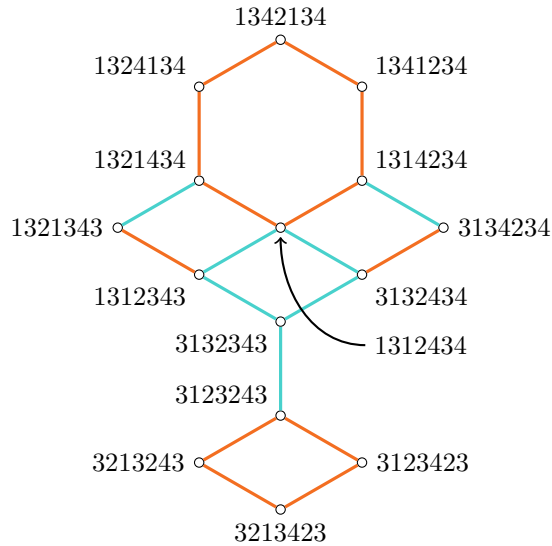
\begin{figure}[!ht]
\begin{center}
\begin{tikzpicture}[every circle node/.style={draw, circle, inner sep=1.25pt}, scale=1.25]
\node [circle] (1) [label=above right:$1341234$] at (30:1){};
\node [circle] (2) [label=above:$1342134$] at (90:1){};
\node [circle] (3) [label=above left:$1324134$]at (150:1){};
\node [circle] (4) [label=above left:$1321434$] at (210:1){};
\node [circle] (5) [] at (0,-1){};
\node (h) [label = right:$1312434$] at (.8,-2.25){};
\node (y) at (0,-1){};
\node (x) at (1,-2.25){};
\node [circle] (6) [label=above right:$1314234$] at (330:1){};
\node [circle] (7) [label=left:$1321343$] at (210:2){};
\node [circle] (8) [label=right:$3134234$] at (330:2){};
\node [circle] (9) [label=below left:$1312343$] at ({-sqrt(3)/2},-1.5){};
\node [circle] (10) [label=below right:$3132434$] at ({sqrt(3)/2},-1.5){};
\node [circle] (11) [label=below left:$3132343$] at (0,-2){};
\node [circle] (12) [label=above left:$3123243$] at (0,-3){};
\node [circle] (13) [label=left:$3213243$] at ({-sqrt(3)/2},-3.5){};
\node [circle] (14) [label=right:$3123423$] at ({sqrt(3)/2},-3.5){};
\node [circle] (15) [label=below:$3213423$] at (0,-4){};

\draw [nectarine,-, very thick] (1) to (2);
\draw [nectarine,-, very thick] (2) to (3);
\draw [nectarine,-, very thick] (3) to (4);
\draw [nectarine,-, very thick] (4) to (5);
\draw [nectarine,-, very thick] (5) to (6);
\draw [nectarine,-, very thick] (6) to (1);
\draw [turq,-, very thick] (4) to (7);
\draw [turq,-, very thick] (6) to (8);
\draw [turq,-, very thick] (5) to (9);
\draw [turq,-, very thick] (5) to (10);
\draw [nectarine,-, very thick] (7) to (9);
\draw [nectarine,-, very thick] (8) to (10);
\draw [turq,-, very thick] (9) to (11);
\draw [turq,-, very thick] (10) to (11);
\draw [turq,-, very thick] (11) to (12);
\draw [nectarine,-, very thick] (12) to (13);
\draw [nectarine,-, very thick] (12) to (14);
\draw [nectarine,-, very thick] (13) to (15);
\draw [nectarine,-, very thick] (14) to (15);
\draw [->, thick] (x) to [out=180,in=270] (y);
\end{tikzpicture}
\setlength{\belowcaptionskip}{-6pt}
\caption{Matsumoto graph for the reduced expression in the Coxeter system of type $D_4$ described in Example~\ref{ex:matsumotoD4}}
\label{fig:Matsumoto graph}
\end{center}
\end{figure}

We now define two different equivalence relations on the set of reduced expressions for a given element of a Coxeter group. Take $(W,S)$ to be a Coxeter system and let $w \in W$. For $\ralpha, \rbeta \in \mathcal{R}(w)$, we define a symmetric relation $\sim_c$ via $\ralpha \sim_c \rbeta$ if $\ralpha$ may be obtained from $\rbeta$ by performing a single commutation move. The equivalence relation $\approx_c$ is defined by taking the reflexive and transitive closure of $\sim_c$ (i.e., $\approx_c$ is the smallest equivalence relation containing $\sim_c$). The corresponding equivalence classes under $\approx_c$ are referred to as \emph{commutation classes}, denoted $[\ralpha]_c$. Appropriately, we say that two reduced expressions are \emph{commutation equivalent} if they are in the same commutation class. 

Analogously, we define $\sim_b$ via $\ralpha \sim_b \rbeta$ if $\ralpha$ may be obtained from $\rbeta$ by applying a single braid move. We define the equivalence relation $\approx_b$ by taking the reflexive and transitive closure of $\sim_b$, and call each equivalence class under $\approx_b$ a \emph{braid class}, denoted $[\ralpha]_b$. If two reduced expressions are in the same braid class, we say that these expressions are \emph{braid equivalent}. 

\begin{example}\label{classes}
Consider the reduced expression $\ralpha=1321434$ in the Coxeter system of type $D_4$ from Example~\ref{ex:matsumotoD4}. The set of 15 reduced expressions is partitioned into five commutation classes and nine braid classes. The braid classes correspond to the connected components of the \textcolor{turq}{teal} subgraph obtained by deleting the \textcolor{nectarine}{orange} edges of the Matsumoto graph given in Figure~\ref{fig:Matsumoto graph}. In particular, the singleton braid classes correspond to the six vertices that are not incident to any \textcolor{turq}{teal} edges.
\end{example}

The remainder of this paper focuses exclusively on braid classes, so we will now write $[\ralpha]$ in place of $[\ralpha]_b$. Each \textcolor{turq}{teal} connected component of a Matsumoto graph provides a graphical representation of the corresponding braid class. For a reduced expression $\ralpha$, the \emph{braid graph} of $\ralpha$, denoted $\B(\ralpha)$, is the graph whose vertex set is  $[\ralpha]$ and $\rbeta, \rgamma \in [\ralpha]$ are connected by an edge if and only if $\rgamma$ and $\rbeta$ are related by a single braid move. If $\ralpha$ and $\rbeta$ are braid equivalent, then $\B(\ralpha) = \B(\rbeta)$. On the other hand, if $\ralpha$ and $\rbeta$ are related via a commutation move, then $\B(\ralpha) \neq \B(\rbeta)$ but they might be isomorphic. Note that we are defining braid graphs with respect to a fixed reduced expression (or equivalence class) as opposed to the corresponding group element. The latter are the graphs that arise from contracting the edges corresponding to braid moves in the Matsumoto graph. 

\begin{example}\label{ex:braid classes}
Below we describe braid classes for three different reduced expressions and their corresponding braid graphs. We have used underlines and overlines to indicate where braid moves may occur.
\begin{enumerate}[(a)] 
\item In the Coxeter system of type $A_6$, the expression $\ralpha_1 =1213243565$ is reduced. Its braid class consists of the following reduced expressions: 
\[
\UOLaugment\overline 
\UOLaugment\underline
\ralpha_1 = \underline{121}3243\underline{565}, \
\ralpha_2 = \underline{21}[2]\overline{32}43\underline{565}, \
\ralpha_3 = 21\underline{32}[3]\overline{43}\underline{565}, \
\ralpha_4 = 2132\overline{434}\underline{565},
\]
\[
\UOLaugment\overline 
\UOLaugment\underline
\ralpha_5 = \underline{121}3243\underline{656}, \
\ralpha_6 = \underline{21}[2]\overline{32}43\underline{656}, \
\ralpha_7 = 21\underline{32}[3]\overline{43}\underline{656}, \
\ralpha_8 = 2132\overline{434}\underline{656}.
\]
\item In the Coxeter system of type $D_4$, the expression $\rbeta_1 = 4341232$ is reduced and its braid class consists of the following reduced expressions: 

\[
\UOLaugment\overline 
\UOLaugment\underline
\rbeta_1 = \underline{434} 1 \underline{232}, \
\rbeta_2 = \underline{343} 1 \underline{232}, \
\rbeta_3 = \underline{434} 1 \underline{323}, \
\rbeta_4 = \underline{34}[3]\overline{1}[3]\underline{23}, \
\rbeta_5 = 34\underline{131}23.
\]
\item In the Coxeter system of type $D_4$, the expression $\rgamma_1 = 343132343$ is reduced and its braid class consists of the following reduced expressions:
\[
\UOLaugment\overline 
\UOLaugment\underline
\rgamma_1 = \underline{34}[3]\overline{1}[3]\underline{2}[3]\overline{43}, \
\rgamma_2 = 34\underline{131}2\underline{343}, \
\rgamma_3 = \underline{434}1\underline{32}[3]\overline{43}, \
\rgamma_4 = \underline{343}1\underline{232}43, \
\]
\[
\UOLaugment\overline 
\UOLaugment\underline
\rgamma_5 = \underline{434}1\underline{232}43, \
\rgamma_6 = \underline{343}132\underline{434}, \
\rgamma_7 = 34\underline{131}2\underline{434}, \
\rgamma_8 = \underline{434}132\underline{434}.
\]
\end{enumerate}
The braid graphs $B(\ralpha_1)$, $B(\rbeta_1)$, and $B(\rgamma_1)$ are depicted in Figure~\ref{fig:braid graphs}.
\end{example}

\begin{figure}[!ht]
\centering
\subcaptionbox{\label{fig:braidgraph_a}}[.28\linewidth]{
\begin{tikzpicture}[every circle node/.style={draw, circle, inner sep=1.25pt}]
\node [circle] (1) [label=left:$\ralpha_4$] at (0,1){};
\node [circle] (2) [label=left:$\ralpha_3$] at (0,2){};
\node [circle] (3) [label=left:$\ralpha_2$] at (0,3){};
\node [circle] (4) [label=left:$\ralpha_1$] at (0,4){};
\node [circle] (5) [label=right:$\ralpha_8$] at (1,1){};
\node [circle] (6) [label=right:$\ralpha_7$] at (1,2){};
\node [circle] (7) [label=right:$\ralpha_6$] at (1,3){};
\node [circle] (8) [label=right:$\ralpha_5$] at (1,4){};
\draw [turq,-, very thick] (1) to (2);
\draw [turq,-, very thick] (2) to (3);
\draw [turq,-, very thick] (3) to (4);
\draw [turq,-, very thick] (5) to (6);
\draw [turq,-, very thick] (6) to (7);
\draw [turq,-, very thick] (7) to (8);
\draw [turq,-, very thick] (1) to (5);
\draw [turq,-, very thick] (2) to (6);
\draw [turq,-, very thick] (3) to (7);
\draw [turq,-, very thick] (4) to (8);
\end{tikzpicture}
}
\subcaptionbox{\label{fig:braidgraph_b}}[.28\linewidth]{
\begin{tikzpicture}[every circle node/.style={draw, circle, inner sep=1.25pt}]
\node [circle] (1) [label=left:$\rbeta_4$] at (0,1){};
\node [circle] (2) [label=left:$\rbeta_5$] at (.707,.293){};
\node [circle] (3) [label=left:$\rbeta_3$] at (-.707,1.707){};
\node [circle] (4) [label=right:$\rbeta_2$] at (.707,1.707){};
\node [circle] (5) [label=left:$\rbeta_1$] at (0,2.414){};
\draw [turq,-, very thick] (1) to (2);
\draw [turq,-, very thick] (1) to (3);
\draw [turq,-, very thick] (1) to (4);
\draw [turq,-, very thick] (3) to (5);
\draw [turq,-, very thick] (5) to (4);
\end{tikzpicture}
}
\subcaptionbox{\label{fig:braidgraph_c}}[.25\linewidth]{
\begin{tikzpicture}[every circle node/.style={draw, circle ,inner sep=1.25pt}]
\node [circle] (1) [label=right:$\rgamma_1$] at (0,1){};
\node [circle] (2) [label=right:$\rgamma_2$] at (.707,.293){};
\node [circle] (3) [label=left:$\rgamma_3$] at (-.707,1.707){};
\node [circle] (4) [label=right:$\rgamma_4$] at (.707,1.707){};
\node [circle] (5) [label=right:$\rgamma_5$] at (0,2.414){};
\node [circle] (6) [label=left:$\rgamma_6$] at (-.707,.293){};
\node [circle] (7) [label=right:$\rgamma_7$] at (0,-0.414){};
\node [circle] (8) [label=left:$\rgamma_8$]at (-1.414,1){};
\draw [turq,-, very thick] (1) to (2);
\draw [turq,-, very thick] (1) to (3);
\draw [turq,-, very thick] (1) to (4);
\draw [turq,-, very thick] (1) to (6);
\draw [turq,-, very thick] (8) to (3);
\draw [turq,-, very thick] (8) to (6);
\draw [turq,-, very thick] (3) to (5);
\draw [turq,-, very thick] (6) to (7);
\draw [turq,-, very thick] (5) to (4);
\draw [turq,-, very thick] (7) to (2);
\draw (1.2,.3) node[label=right:\phantom{$\ralpha_1$}]{};
\end{tikzpicture}}
\caption{Braid graphs corresponding to Example~\ref{ex:braid classes}.\label{fig:braid graphs}}
\end{figure}
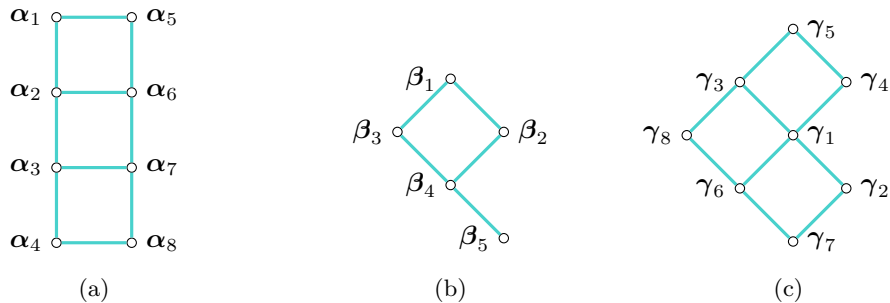

The next proposition is a direct result of Proposition~\ref{prop:Masumoto bipartite}.

\begin{proposition}\label{braidbi}
If $(W,S)$ is a Coxeter system and $\ralpha$ is a reduced expression for $w \in W$, then $\B(\ralpha)$ is bipartite. 
\end{proposition}

\section{Architecture of braid graphs}\label{sec:architecture of braid graphs}

Throughout the remainder of this paper, we will assume $(W,S)$ is simply laced. The following terminology and definitions allow us to introduce the notions of braid shadow and link, which first appeared in~\cite{ABCE2024} . For $i,j \in \mathbb{N}$ with $i \leq j$, we define the interval $\llb i,j \rrb := \{i, i+1, \ldots, j-1,j\}$. It follows that $\llb i, i \rrb = \{i\}$. The intervals $\llb i,j \rrb$ will be used to denote a contiguous set of positions of a reduced expression. 

For a reduced expression $\ralpha = s_{x_1}s_{x_2} \cdots s_{x_m}$, the \emph{local support} of $\ralpha$ over the interval $\llb i,j \rrb$ is defined via
\[
\supp_{\llb i,j \rrb}(\ralpha) := \{s_{x_k} \mid k \in \llb i,j \rrb\}.
\]
We define the \emph{local support of the braid class} $[\ralpha]$ over the interval $\llb i,j \rrb$ via
\[
\supp_{\llb i,j \rrb}([\ralpha]) := \bigcup_{\rbeta \in [\ralpha]} \supp_{\llb i,j \rrb}(\rbeta).
\]
That is, the set $\supp_{\llb i,j \rrb} (\ralpha)$ contains the generators that appear in positions $i,i+1, \ldots, j$ of a single reduced expression $\ralpha$, while $\supp_{\llb i,j \rrb}([\ralpha])$ contains the generators that appear in positions $i, i+1, \ldots ,j$ of some reduced expression in the braid class $[\ralpha]$. In the special case of the degenerate interval $\llb i, i \rrb$, we write $\supp_i (\ralpha):= \supp_{\llb i,i \rrb}(\ralpha)$ and $\supp_i([\ralpha]) := \supp_{\llb i,i \rrb}([\ralpha])$. Further, we let $\ralpha_{\llb i,j \rrb}$ denote the factor of $\ralpha$ appearing in positions $i,i+1, \ldots ,j$ of $\ralpha$.  

Let $\ralpha = s_{x_1}s_{x_2} \cdots s_{x_m}$ be a reduced expression for $w \in W$. Following~\cite{ABCE2024}, we say $\llb i-1,i+1 \rrb$ is a \emph{braid shadow} for $\ralpha$ if $\ralpha_{\llb i-1,i+1 \rrb} = sts$ with $m(s,t) = 3$. That is, a braid shadow is the triple of locations where one may apply a braid move. We denote the collection of braid shadows for $\ralpha$ by $\mathcal{S}(\ralpha)$. The set of braid shadows for the braid class $[\ralpha]$ is aptly defined as
\[
\mathcal{S}([\ralpha]) := \bigcup_{\rbeta \in [\ralpha]} \mathcal{S}(\rbeta).
\]
Note that $\mathcal{S}(\ralpha)$ is the collection of all braid shadows for a specific reduced expression $\ralpha$, while $\mathcal{S}([\ralpha])$ is the set of braid shadows for all reduced expressions braid equivalent to $\ralpha$.
The \emph{dimension} of a reduced expression $\ralpha$, denoted $\dim(\ralpha)$, is defined to be the cardinality of $\mathcal{S}([\ralpha])$. In~\cite{ABCE2024}, the authors used the term ``rank'' instead of ``dimension''.

\begin{example}\label{ex:braidshadows}
Consider the reduced expressions given in Example~\ref{ex:braid classes}. We see that:
\begin{enumerate}[(a)]
\item $\mathcal{S}(\ralpha_1) = \{\llb 1,3 \rrb ,\llb  8,10 \rrb\}$, $\mathcal{S}([\ralpha_1]) = \{\llb 1,3 \rrb,\llb 3,5 \rrb,\llb  5,7 \rrb,\llb  8,10 \rrb\}$, and $\dim(\ralpha_1)=4$;
\item $\mathcal{S}(\rbeta_1) = \{\llb 1,3 \rrb,\llb 5,7 \rrb\}$, $\mathcal{S}([\rbeta_1]) = \{\llb 1,3 \rrb,\llb 3,5 \rrb,\llb 5,7 \rrb\}$, and $\dim(\rbeta_1)=3$;
\item $\mathcal{S}(\rgamma_1) = \{\llb 1,3 \rrb,\llb 3,5 \rrb,\llb 5,7 \rrb,\llb  7,9 \rrb\}$, $\mathcal{S}([\rgamma_1]) = \{\llb 1,3 \rrb,\llb 3,5 \rrb,\llb 5,7 \rrb,\llb 7,9 \rrb\}$, and $\dim(\rgamma_1)=4$.
\end{enumerate}
\end{example}

The following result from~\cite{ABCE2024} states that for any reduced expression $\ralpha$, any distinct pair of braid shadows across $[\ralpha]$ must either be disjoint or overlap by exactly one position.

\begin{proposition}
Suppose $(W,S)$ is a simply-laced Coxeter system. If $\ralpha$ is a reduced expression for $w \in W$ with $\llb i-1, i+1 \rrb \in \mathcal{S}([\ralpha])$, then $\llb i-2, i \rrb, \llb i, i+2 \rrb \notin \mathcal{S}([\ralpha])$.
\end{proposition}

The previous result inspires the following definition from~\cite{ABCE2024}. If  $\ralpha$ is a reduced expression for $w \in W$ with $\ell(w) = m \geq 1$, we define $\ralpha$ to be a \emph{link} if either $m = 1$ or $m$ is odd and 
\[
\mathcal{S}([\ralpha]) = \{\llb 1,3 \rrb, \llb 3,5 \rrb, \ldots ,\llb m-4,m-2 \rrb, \llb m-2,m \rrb\}.
\] 
Note that if $\ralpha$ is a link and $\rbeta \in [\ralpha]$, then $\rbeta$ is also a link.

\begin{example}\label{ex:linkschains}
Consider the reduced expressions given in Example~\ref{ex:braid classes}.  Since $\llb 7,9 \rrb \notin \mathcal{S}([\ralpha_1])$ and $m$ is not odd, $\ralpha_1$ is a not a link. It turns out that the factors $1213243$ and $565$ of $\ralpha_1$ are links. However, since $\mathcal{S}([\rbeta_1]) = \{\llb  1,3 \rrb,\llb  3,5 \rrb,\llb  5,7 \rrb\}$, $\rbeta_1$ is a link. Lastly, since  $\mathcal{S}([\rgamma_1])=\{\llb  1,3 \rrb,\llb  3,5 \rrb,\llb  5,7 \rrb,\llb  7,9 \rrb\}$, $\rgamma_1$ is also a link.
\end{example}

Let $\ralpha$ be a reduced expression for $w \in W$ such that $\ell(w) \geq 1$. Then $\rbeta$ is said to be a \emph{link factor} of $\ralpha$ if and only if
\begin{enumerate}[(a)]
\item $\rbeta$ is a factor of $\ralpha$,
\item $\rbeta$ is a link, and
\item If $\rbeta < \rgamma \leq \ralpha$, then $\rgamma$ is not a link.
\end{enumerate} 
That is, the link factors of a reduced expression are maximal among the factors of that expression that are also links. It follows that we may uniquely write each reduced expression $\ralpha$ for a nonidentity group element as a product $\ralpha_1\ralpha_2 \cdots \ralpha_k$, where each $\ralpha_i$ is a link factor. This product is called the \emph{link factorization} of $\ralpha$. We may denote the link factorization as $\ralpha = \ralpha_1 \mid \ralpha_2 \mid \cdots \mid \ralpha_k$. For convenience, we say that the link factorization of the identity is a product consisting of a single copy of the empty word despite the fact that the empty word is not a link. The next result appears in~\cite{ABCE2024}.

\begin{proposition}\label{prop:braid classes simply laced}
Suppose $(W,S)$ is a simply-laced Coxeter system. If $\ralpha$ is a reduced expression for $w \in W$ with link factorization $\ralpha_1 \mid \ralpha_2 \mid \cdots \mid \ralpha_k$, then
\begin{enumerate}[(a)]
\item $[\ralpha] = \{\rbeta_1 \mid \rbeta_2 \mid \cdots \mid \rbeta_k : \rbeta_i \in [\ralpha_i] \text{ for } 1 \leq i \leq k\},$
\item $\card([\ralpha]) = \displaystyle \prod_{i=1}^k \card([\ralpha_i]),$
\item $\displaystyle\dim(\ralpha) = \sum_{i=1}^k \dim(\ralpha_i)$.
\item\label{box prod} $\B(\ralpha) \cong \B(\ralpha_1) \square \B(\ralpha_2) \square \cdots \square \B(\ralpha_k)$.
\end{enumerate}
\end{proposition}

\begin{example}\label{ex:boxprod2}
\UOLaugment\overline 
\UOLaugment\underline
Consider the reduced expression $\ralpha = 3231343565787$ in the Coxeter system of type $D_7$.  The link factorization for $\ralpha$ is $\textcolor{turq}{3231343} \mid \textcolor{magenta}{565} \mid \textcolor{nectarine}{787}$. The braid graph for the first link factor is isomorphic to the braid graph in Figure~\ref{fig:braidgraph_b}. The braid graph for the entire reduced expression and its decomposition are shown in Figure~\ref{fig:productoflollipops}. We have utilized colors to help distinguish the link factors.
\end{example}

\begin{figure}[ht!]
\centering
\UOLaugment\overline 
\UOLaugment\underline
\begin{tikzpicture}[every circle node/.style={draw, circle, inner sep=1.25pt}]
\node [circle] (1) at (1.3,0){};
\node [label=above right:{$\textcolor{turq}{\underline{32}[3]\overline{1}[3]\underline{43}} \mid \textcolor{magenta}{\overline{565}} \mid \textcolor{nectarine}{\underline{787}}$}] (0) at (1.2,-.45){};
\node [circle] (2) at ({1.3*cos(45)},{1.3*sin(45)}){};
\node [circle] (3) at (0,1.3){};
\node [circle] (4) at ({1.3*cos(135)},{1.3*sin(135)}){};
\node [circle] (5) at (-1.3,0){};
\node [circle] (6) at ({1.3*cos(225)},{1.3*sin(225)}){};
\node [circle] (7) at (0,-1.3){};
\node [circle] (8) at ({1.3*cos(315)},{1.3*sin(315)}){};
\draw [turq,-, very thick] (1) to (2);
\draw [turq,-, very thick] (2) to (3);
\draw [magenta,-, very thick] (3) to (4);
\draw [nectarine,-, very thick] (4) to (5);
\draw [turq,-, very thick] (5) to (6);
\draw [turq,-, very thick] (6) to (7);
\draw [magenta,-, very thick] (7) to (8);
\draw [nectarine,-, very thick] (8) to (1);

\node [circle] (9) at (.55,0){};
\node [circle] (10) at ({.55*cos(45)},{.55*sin(45)}){};
\node [circle] (11) at (0,.55){};
\node [circle] (12) at ({.55*cos(135)},{.55*sin(135)}){};
\node [circle] (13) at (-.55,0){};
\node [circle] (14) at ({.55*cos(225)},{.55*sin(225)}){};
\node [circle] (15) at (0,-.55){};
\node [circle] (16) at ({.55*cos(315)},{.55*sin(315)}){};

\draw [turq,-, very thick] (8) to (9);

\draw [magenta,-, very thick] (1) to (16);
\draw [nectarine,-, very thick] (2) to (9);
\draw [turq,-, very thick] (3) to (10);

\draw [magenta,-, very thick] (5) to (12);
\draw [nectarine,-, very thick] (6) to (13);

\draw [turq,-, very thick] (8) to (15);
\draw [turq,-, very thick] (1) to (10);
\draw [magenta,-, very thick] (2) to (11);
\draw [nectarine,-, very thick] (3) to (12);
\draw [turq,-, very thick] (4) to (13);
\draw [turq,-, very thick] (5) to (14);
\draw [magenta,-, very thick] (6) to (15);
\draw [nectarine,-, very thick] (7) to (16);

\draw [turq,-, very thick] (9) to (12);
\draw [magenta,-, very thick] (10) to (13);

\draw [turq,-, very thick] (12) to (15);
\draw [turq,-, very thick] (13) to (16);
\draw [magenta,-, very thick] (14) to (9);
\draw [nectarine,-, very thick] (15) to (10);
\draw [turq,-, very thick] (16) to (11);

\node [circle] (17) at (1.3+.92,-.393){};
\node [circle] (18) at (0+.92,-1.3-.393){};
\node [circle] (19) at ({1.3*cos(315)+.92},{1.3*sin(315)-.393}){};
\node [circle] (20) at ({.55*cos(315)+.92},{.55*sin(315)-.393}){};

\draw [turq,-, very thick] (19) to (8);
\draw [turq,-, very thick] (17) to (1);
\draw [turq,-, very thick] (18) to (7);
\draw [turq,-, very thick] (20) to (16);

\draw [nectarine,-, very thick] (17) to (19);
\draw [magenta,-, very thick] (18) to (19);
\draw [magenta,-, very thick] (17) to (20);
\draw [nectarine,-, very thick] (20) to (18);

\draw [turq,-, very thick] (7) to (14);

\draw [nectarine,-, very thick] (11) to (14);
\draw [turq,-, very thick] (4) to (11);

\def\lshiftA{-9.3}

\node (a) at (-2.5,-.504) {$\cong$};

\node [circle] [label=left:{$\textcolor{magenta}{\overline{565}}$}](26) at (4 + \lshiftA ,-.504 + .45){};
\node [circle] (27) at (4 + \lshiftA ,-.504 - .45){};
\draw [magenta,-, very thick] (26) to (27);

\node (b) at (4.75 + \lshiftA,-.504) {$\Box$};

\node [circle] [label=left:{$\textcolor{nectarine}{\underline{787}}$}] (28) at (5.5 + \lshiftA ,-.504 + .45){};
\node [circle] (29) at (5.5 + \lshiftA ,-.504 - .45){};
\draw [nectarine,-, very thick] (28) to (29);

\node (c) at (3 + \lshiftA,-.5) {$\Box$};

\def\lshiftB{1.25}

\node [circle] (33) [label=left: {$\textcolor{turq}{\underline{32}[3]\overline{1}[3]\underline{43}}\phantom{1}$}] at (\lshiftA + \lshiftB,1-1.5){};
\node [circle] (34) at (.707+\lshiftA + \lshiftB,.293-1.5){};
\node [circle] (35) at (-.707+\lshiftA+ \lshiftB,1.707-1.5){};
\node [circle] (36) at (.707+ \lshiftA+ \lshiftB,1.707-1.5){};
\node [circle] (37) at (\lshiftA+ \lshiftB,2.414-1.5){};
\draw [turq,-, very thick] (33) to (34);
\draw [turq,-, very thick] (33) to (35);
\draw [turq,-, very thick] (33) to (36);
\draw [turq,-, very thick] (35) to (37);
\draw [turq,-, very thick] (37) to (36);
\end{tikzpicture}
\caption{Braid graph for the reduced expression from Example~\ref{ex:boxprod2} and its decomposition into a box product of braid graphs for the corresponding link factors.}\label{fig:productoflollipops}
\end{figure}
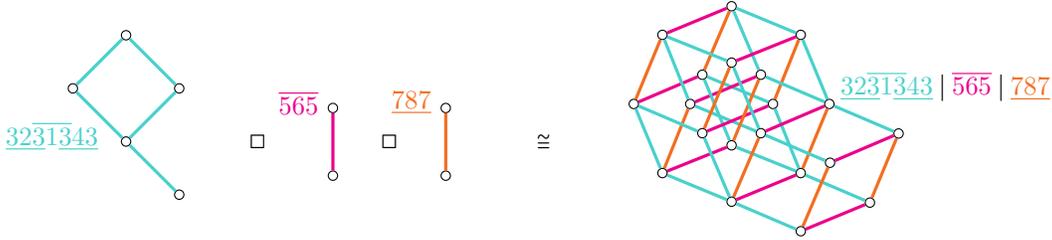

We now focus our attention on Coxeter systems of type $\Lambda$.  The next several propositions summarize key results from~\cite{ABCE2024} concerning the local structure of links in Coxeter systems of type $\Lambda$.

\begin{proposition}\label{prop:local structure of braid shadows}
Suppose $(W,S)$ is type $\Lambda$ and let $\ralpha$ be a link of dimension $r \geq 1$. For each $1\leq i\leq r$, there exists unique $s,t\in S$ with $m(s,t)=3$ such that if $\rbeta\in [\ralpha]$ with $\llb 2i-1,2i+1 \rrb \in \mathcal{S}(\rbeta)$, then $\rbeta_{\llb 2i-1,2i+1 \rrb}=sts$ or $\rbeta_{\llb 2i-1,2i+1 \rrb}=tst$ and $\supp_{2i}([\ralpha])=\{s,t\}$.
\end{proposition}

If $\ralpha$ is a reduced expression such that $\llb i-1, i+1 \rrb$ is a braid shadow for $[\ralpha]$, then the position $i$ in any reduced expression in $[\ralpha]$ is called the \emph{center} of the braid shadow regardless of whether one may apply a braid move in that location in a given reduced expression. The previous proposition states that the support of a braid shadow determines which generators may appear at the corresponding center across the entire braid class. In particular, in Coxeter systems of type $\Lambda$, the center of each braid shadow may take on one of two possible values. Proposition~\ref{prop:local structure of braid shadows} implies that if $\ralpha$ is a link of dimension at least one in a Coxeter system of type $\Lambda$ such that $\supp_{2}([\ralpha])=\{s,t\}$ (with $m(s,t)=3$), then $\ralpha_{\llb1,2\rrb}=st$ or $\ralpha_{\llb1,2\rrb}=ts$. Certainly, we have a similar statement for the right end of a link.

\begin{proposition}\label{prop:local structure of overlapping braid shadows}
Suppose $(W,S)$ is type $\Lambda$ and let $\ralpha$ be a link of dimension $r \geq 2$. For each $1\leq i\leq r$, there exists unique $s,t,u\in S$ with $m(s,t)=3=m(t,u)$ and $m(s,u)=2$ such that $\supp_{2i}([\ralpha])=\{s,t\}$ and $\supp_{2i+2}([\ralpha])=\{t,u\}$.  Additionally, there are three possible forms that $\ralpha_{\llb  2i,2i+2 \rrb}$ may take:
\begin{enumerate}[(i)]
\item $\displaystyle\underbrace{\cdots\lfrac{?}{2i-1}\lfrac{s}{2i}\lfrac{u}{2i+1}\lfrac{t}{2i+2}\lfrac{?}{2i+3}\cdots}_{\ralpha}$,
\item $\displaystyle\underbrace{\cdots\lfrac{?}{2i-1}\lfrac{s}{2i}\lfrac{t}{2i+1}\lfrac{u}{2i+2}\lfrac{?}{2i+3}\cdots}_{\ralpha}$,
\item $\displaystyle\underbrace{\cdots\lfrac{?}{2i-1}\lfrac{t}{2i}\lfrac{s}{2i+1}\lfrac{u}{2i+2}\lfrac{?}{2i+3}\cdots}_{\ralpha}$.
\end{enumerate}
\end{proposition}

The next example illustrates the need for the triangle-free assumption in both Propositions~\ref{prop:local structure of braid shadows} and~\ref{prop:local structure of overlapping braid shadows}.

\begin{example}
Consider the link $\rdelta_1 = 1213121$ in the Coxeter system of type $\widetilde A_2$, which is determined by the Coxeter graph in Figure~\ref{fig:CoxgraphaffA_n}. The braid class for $\rdelta_1$ consists of the following links:
\[
\UOLaugment\overline 
\UOLaugment\underline
\rdelta_1 = \underline{12}[1]\overline{{3}}[1]\underline{21}, \
\rdelta_2 = 12\overline{3{1}3}21, \
\rdelta_3 = \underline{212}{3}\underline{121}, \
\]
\[
\UOLaugment\overline 
\UOLaugment\underline
\rdelta_4 = \underline{12}[1]\overline{{3}}[2]\underline{12}, \
\rdelta_5 = \underline{21}[2]\overline{{3}}[2]\underline{12}, \
\rdelta_6 = 21\overline{3{2}3}12. \
\]
Notice that $\supp_{\llb 3,5 \rrb}(\rdelta_1) = \{1,3\}$, $\supp_{\llb 3,5 \rrb}(\rdelta_5)= \{2,3\}$, and $\supp_4([\rdelta_1])=\{1,2,3\}$, and hence Propositions~\ref{prop:local structure of braid shadows} and~\ref{prop:local structure of overlapping braid shadows} do not hold unless the Coxeter system is triangle free.
\end{example}

The \emph{signature} of a reduced expression $\ralpha$, denoted $\sig(\ralpha)$, is the ordered list of generators of $\ralpha$ appearing in the centers of the braid shadows of $[\ralpha]$. Note that if $\ralpha$ is a link, then $\sig(\ralpha)$ is the ordered list of generators appearing in the even positions. We use $\sig_i(\ralpha)$ to represent the $i$th entry of $\sig(\ralpha)$. In light of Proposition~\ref{prop:local structure of braid shadows}, each $\sig_i(\ralpha)$ takes on one of the two values of the support of the corresponding center. We also define
\[
\barsig_i(\ralpha):=\{ \rx \in [\ralpha] \mid \sig_i(\rx)=\sig_i(\ralpha)\}.
\] 
In other words, $\barsig_i(\ralpha)$ is the set of reduced expressions that are braid equivalent to $\ralpha$ and have the same generator in the center of the $i$th braid shadow as $\ralpha$. One consequence of Proposition~\ref{prop:local structure of overlapping braid shadows} is that in Coxeter systems of type $\Lambda$, adjacent values of the signature are never the same.

\begin{example}\label{ex:signature}
Consider the braid class for $\rgamma_1 = 343132343$ in the Coxeter system of type $D_4$ from Example~\ref{ex:braid classes}(c).  We see that $\sig(\rgamma_1)=(4,1,2,4)$ and $\barsig_4(\rgamma_1)=\{\rgamma_1,\rgamma_2,\rgamma_3,\rgamma_4,\rgamma_5\}$.
\end{example}

The next proposition states that each link in a Coxeter system of type $\Lambda$ is uniquely determined by its signature. This result originally appeared in~\cite{ABCE2024}, but we have rephrased it in terms of signature.

\begin{proposition}\label{prop:equal iff sig equal}
Suppose $(W,S)$ is type $\Lambda$ and let $\ralpha$ and $\rbeta$ be two braid equivalent links of dimension at least one. Then $\ralpha = \rbeta$ if and only if $\sig(\ralpha) = \sig(\rbeta)$.
\end{proposition}

The next proposition from~\cite{ABCE2024} states that if $\ralpha$ is a link, then for all pairs of overlapping braid shadows in $[\ralpha]$, there is a link in $[\ralpha]$ where overlapping braid shadows occur simultaneously.

\begin{proposition}\label{prop:special link}
If $(W,S)$ is type $\Lambda$ and $\ralpha$ is a link of dimension $r \geq 2$, then for all $1\leq i\leq r-1$, there exists $\rsigma \in [\ralpha]$ with the property that $\llb 2i-1,2i+1 \rrb, \llb 2i+1,2i+3 \rrb \in \mathcal{S}(\rsigma)$.
\end{proposition}

In light of Propositions~\ref{prop:local structure of overlapping braid shadows} and~\ref{prop:special link}, if $\ralpha$ is a link of dimension $r \geq 2$, then for each $1\leq i\leq r-1$, there exists $\rsigma \in [\ralpha]$ with the property that $\rsigma_{\llb  2i-1,2i+3 \rrb}=tstut$ for unique $s,t,u\in S$ with $m(s,t)=3=m(t,u)$ and $m(s,u)=2$.

For a reduced expression $\ralpha$, we number the braid shadows in $\bs([\ralpha])$ from left to right as 1 through $\dim(\ralpha)$. If $\ralpha$ and $\rbeta$ are related by a single braid move that occurs in the $j$th braid shadow (i.e., only the $j$th entry of the signature differs between $\ralpha$ and $\rbeta$), we denote this braid move as $b^j$ and write $b^j(\ralpha) = \rbeta$ to indicate that applying the braid move $b^j$ to $\ralpha$ yields $\rbeta$. We accordingly label the edge in $\B(\ralpha)$ connecting $\ralpha$ and $\rbeta$ with a $j$.

As a special case of Proposition~\ref{prop:special link}, if we choose $\rsigma \in [\ralpha]$ according to Proposition~\ref{prop:special link} with $\llb 2r-3,2r-1 \rrb,\llb 2r-1,2r+1 \rrb \in \mathcal{S}(\rsigma)$, then $\barsig_r(\rsigma)$ is the set of links in $[\ralpha]$ that share the same penultimate generator as $\ralpha$ while $\barsig_r(b^r(\rsigma))$ is the set of links that do not. In fact, all links in $\barsig_r(\rsigma)$ have the same final two generators and all links in $\barsig_r(b^r(\rsigma))$ have the same final two generators. Specifically, if $\supp_{2r}([\ralpha]) = \{s,t\}$ and $\supp_{2r}(\ralpha) = \{s\}$, then every link in $\barsig_r(\rsigma)$ ends in $st$ while every link in $\barsig_r(b^r(\rsigma))$ ends in $ts$.  The sets $\barsig_r(\rsigma)$ and $\barsig_r(b^r(\rsigma))$ will play an important role in the remainder of this paper. Note that $\barsig_r(\rsigma)$ and $\barsig_r(b^r(\rsigma))$ were respectively denoted by $X_\rsigma$ and $Y_\rsigma$ in~\cite{ABCE2024}.
 
Let $\ralpha$ be a link of dimension at least 1. We define $\hatralpha$ to be the reduced expression obtained by deleting the two rightmost letters of $\ralpha$. Certainly, $\hatralpha$ is reduced, but it is important to note that $\hatralpha$ may no longer be a link. The subsequent proposition, which combines multiple results from~\cite{ABCE2024}, establishes a sufficient condition on a link $\ralpha$ such that $\hatralpha$ is also a link. 

\begin{proposition}\label{prop:multiple results cobbled}
Suppose $(W,S)$ is type $\Lambda$ and $\ralpha$ is a link of dimension $r\geq 2$ and choose $\rsigma$ according to Proposition~\ref{prop:special link} such that $\llb 2r-3,2r-1 \rrb, \llb 2r-1,2r+1 \rrb \in \mathcal{S}(\rsigma)$. Then: 
\begin{enumerate}[(a)]
\item $\{\barsig_r(\rsigma),\barsig_r(b^r(\rsigma))\}$ is a partition of $[\ralpha]$;
\item $\hatrsigma$ is a link of dimension $r-1$;
\item If $\rbeta \in \barsig_r(\rsigma)$, then $\hatrbeta \in [\hatrsigma]$;
\item Every element of $[\hatrsigma]$ is of the form $\hatrbeta$ for some $\rbeta \in \barsig_r(\rsigma)$;
\item\label{cor:subgraph of a link a} There exists an isometric embedding from $\B(\hatrsigma)$ into $\B(\ralpha)$ whose image is $\B(\ralpha)[\barsig_r(\rsigma)]$;
\item If $\rbeta \in \barsig_r(b^r(\rsigma))$, then $\llb 2r-1,2r+1 \rrb \in \mathcal{S}(\rbeta)$ and $\left(b^r(\rbeta)\right)_{\llb 1,2r-1 \rrb} \in [\hatrsigma]$.
\item\label{cor:subgraph of a link b} The induced subgraph $B(\ralpha)[\barsig_r(b^r(\rsigma))]$ is an isometric subgraph of $B(\ralpha)$;
\item\label{cor:subgraph of a link c} If $\rbeta \in \barsig_r(\rsigma)$ and  $\rgamma \in \barsig_r(b^r(\rsigma))$, then $d(\rbeta ,\rgamma) = d(\rbeta ,b^r(\rgamma)) + 1$.
\end{enumerate}
\end{proposition}

\begin{example}\label{ex:braid graph with partition}
Consider the link $\ralpha = 32313435464$ in the Coxeter system of type $\widetilde D_5$. One possible choice for a link satisfying the conditions in Proposition~\ref{prop:special link} with braid shadows $\llb 7,9 \rrb$ and $\llb 9,11 \rrb$ is $\rsigma= 32314345464$. Set $\rtau =b^5(\rsigma)=32314345646$.  The braid graph for $\ralpha$ is given in Figure~\ref{fig:BoxWithTwoFlaps}. We have highlighted $\B(\ralpha)[\barsig_5(\rsigma)]$ in \textcolor{turq}{teal} and $\B(\ralpha)[\barsig_5(\rtau)]$ in \textcolor{magenta}{magenta}. In this case, $\hatrsigma=323143454$, and $\B(\hatrsigma)\cong \B(\ralpha)[\barsig_5(\rsigma)]$. Each of the edges joining $\B(\ralpha)[\barsig_5(\rsigma)]$ and $\B(\ralpha)[\barsig_5(\rtau)]$ correspond to the braid move $b^5$ and are shown in black. 
\end{example} 

\begin{figure}[ht!]
\centering
\begin{tikzpicture}[every circle node/.style={draw, circle, inner sep=1.25pt}, scale = 1.2]
	\node at (-2.25,2) {$\textcolor{magenta}{\B(\ralpha)[\barsig_5(\rtau)]}$};
	\node at (-1.2,3.75) {$\textcolor{turq}{\B(\hatrsigma)\cong\B(\ralpha)[\barsig_5(\rsigma)]}$};
	\node (x) at (-2,1) {$\rtau=\underline{323}1\underline{434}5\underline{646}$};
		\UOLaugment\overline 
	\UOLaugment\underline
\node [circle] (1) at (0,1){};
\node [circle] (2) at (.707,.293){};
\node [circle] (3) at (-.707,1.707){};
\node [circle] (4) [label=right:$\ {\underline{32}[3]\overline{1}[3]\underline{43}5\underline{464}=\ralpha}$] at (.707,1.707){};
\node [circle] (5) at (0,2.414){};
\draw [rotate = 90, magenta,-, very thick] (1) to (2);
\draw [magenta,-, very thick] (1) to (3);
\draw [black,-, very thick] (1) to (4);
\draw [black,-, very thick] (3) to (5);
\draw [turq,-, very thick] (5) to (4);
\node [circle] (9) at (1.414,1){};
\node (y) at (-.05,1.707){};
\node [circle] (10) at (0,1.707){};
\node [circle] (11) at (0,3.121){};
\node [circle] (13) [label=right:$\ {\underline{323}1\underline{43}[4]\overline{5}[4]\underline{64}=\rsigma}$] at (.707,2.414){};
\node [circle] (14) at (-.707,2.414){};
\node [circle] (15) at (1.414,3.121){};
\node [circle] (16) at (.707,3.828){};
\draw [magenta,-, very thick] (1) to (10);
\draw [black,-, very thick] (10) to (13);
\draw [magenta,-, very thick] (10) to (14);
\draw [black,-, very thick] (14) to (11);
\draw [turq,-, very thick] (13) to (11);
\draw [turq,-, very thick] (13) to (4);
\draw [turq,-, very thick] (4) to (9);
\draw [black,-, very thick] (9) to (2);
\draw [turq,-, very thick] (5) to (11);
\draw [magenta,-, very thick] (14) to (3);
\draw [turq,-, very thick] (15) to (16);
\draw [turq,-, very thick] (11) to (16);
\draw [turq,-, very thick] (15) to (13);
\draw [->, thick] (x) to [out=0,in=180] (y);
\begin{pgfonlayer}{background}
\highlight{8pt}{turq}{(9.center) to (4.center) to (5.center) to (11.center) to (16.center) to (15.center) to (13.center) to (11.center)}
\highlight{8pt}{turq}{(13.center) to (4.center)}
\highlight{8pt}{magenta}{(2.center) to (1.center) to (3.center) to (14.center) to (10.center) to (1.center)}
\end{pgfonlayer}
\end{tikzpicture}
\caption{Braid graph for the reduced expression in Example~\ref{ex:braid graph with partition} together with a partition of the vertices according to Proposition~\ref{prop:multiple results cobbled}.}
\label{fig:BoxWithTwoFlaps}
\end{figure}
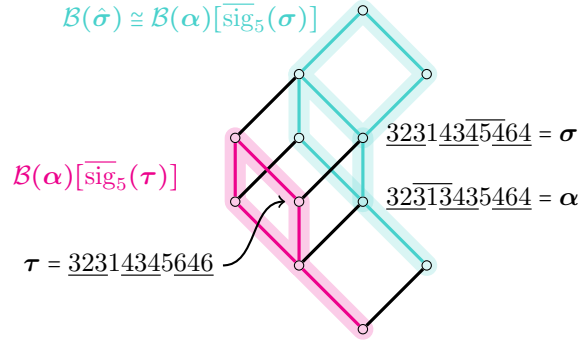

\section{Geodetic structure of braid graphs}\label{sec:geodetic structure}

If $\ralpha$ and $\rbeta$ are braid equivalent reduced expressions in a Coxeter system of type $\Lambda$, we denote a minimal sequence of braid moves from $\ralpha$ to $\rbeta$ as $b_1^{j_1}, b_2^{j_2}, \ldots, b_k^{j_k}$, where $b_i^{j_i}$ is the $i$th braid move in the sequence that occurs in the $j_i$th shadow in $\bs([\ralpha])$. A minimal braid sequence from $\ralpha$ to $\rbeta$ corresponds to a geodesic in $\B(\ralpha)$ from $\ralpha$ to $\rbeta$ consisting of edges labeled consecutively $j_1, j_2, \ldots , j_k$. 

\begin{proposition}\label{prop:each edge once in minimal sequence}
Suppose $(W,S)$ is type $\Lambda$ and let $\ralpha$ and $\rbeta$ be two braid equivalent reduced expressions of dimension at least one.  A braid sequence $b_1^{j_1},b_2^{j_2},...,b_k^{j_k}$ from $\ralpha$ to $\rbeta$ is minimal if and only if each $j_i$ appears exactly once.
\end{proposition}

\begin{proof}
Suppose to the contrary that there exists a minimal braid sequence $b_1^{j_1},b_2^{j_2},...,b_k^{j_k}$ from $\ralpha$ to $\rbeta$ where $j_i=j_{i^*}$ for some $i\neq i^*$. Choose $\ralpha$ and $\rbeta$ such that $k$ is minimal among all such pairs. Since $k$ is minimal, we can assume that $j_1 = j_k$ and this is the only repeated braid move in the sequence. There are two cases.

First suppose that the braid shadows centered at $2j_1$ and $2j_2$ are disjoint. Certainly, $b_1^{j_1}$ and $b_2^{j_2}$ commute, so that $b_2^{j_2},b_1^{j_1},\cdots,b_k^{j_1}$ is also a minimal braid sequence from $\ralpha$ to $\rbeta$. This contradicts the minimality of $k$. 

Now suppose the braid shadows centered at $2j_1$ and $2j_2$ overlap. Without loss of generality, assume that $j_2 = j_1 + 1$. By Proposition~\ref{prop:local structure of overlapping braid shadows}, it must be the case that $b_1^{j_1}(\ralpha)_{\llb 2j_1-1,2j_1+3\rrb}=tstut$, where $m(s,t)=3=m(t,u)$ and $m(s,u)=2$. This implies that $b_2^{j_2}b_1^{j_1}(\ralpha)_{\llb 2j_1-1,2j_1+3\rrb}=tsutu$. Let $\ralpha'$ denote the reduced expression obtained from $\ralpha$ after applying the braid moves $b_1^{j_1},b_2^{j_2},\cdots ,b_{k-1}^{j_{k-1}}$. Since ${j_1}$ and ${j_2}$ are  distinct from $j_3,\ldots,j_{k-1}$ (since we have assumed $j_1=j_k$ is only one repeated pair in this sequence), we see that $\ralpha'_{\llb 2j_1-1,2j_1+1\rrb}=tsu$. Thus, $\llb2j_1-1,2j_1+1\rrb \not\in \bs(\ralpha')$. However, this contradicts the fact that $b_k^{j_1}$ is the final braid move in the sequence.

For the converse, assume that each $j_i$ appears exactly once in a braid sequence $b_1^{j_1},b_2^{j_2},...,b_k^{j_k}$ from $\ralpha$ to $\rbeta$. Then $k$ is at most the number of generators for which $\sig(\ralpha)$ and $\sig(\rbeta)$ differ. On the other hand, a geodesic from $\ralpha$ to $\rbeta$ must include at least one braid move for each position at which the signatures of $\ralpha$ and $\rbeta$ differ. This implies that $k\leq d(\ralpha,\rbeta)$, and so $b_1^{j_1},b_2^{j_2},...,b_k^{j_k}$ must be a minimal braid sequence.
\end{proof}

The next result states that every minimal braid sequence from $\ralpha$ to $\rbeta$ uses the same set of braid shadows.

\begin{proposition}\label{prop:same set in minimal sequence}
Suppose $(W,S)$ is type $\Lambda$ and let $\ralpha$ and $\rbeta$ be two braid equivalent reduced expressions of dimension at least one. If $b_1^{j_1},b_2^{j_2},\ldots,b_k^{j_k}$ and $b_1^{l_1},b_2^{l_2},\ldots,b_k^{l_k}$ are minimal braid sequences from $\ralpha$ to $\rbeta$, then $\{j_1,...,j_k\}=\{l_1,...,l_k\}$.
\end{proposition}

\begin{proof}
For $\ralpha\neq \rbeta$, assume that $b_1^{j_1},b_2^{j_2},\ldots,b_k^{j_k}$ and $b_1^{l_1},b_2^{l_2},\ldots,b_k^{l_k}$ are minimal braid sequences from $\ralpha$ to $\rbeta$. Let $1\leq i \leq k$. By Proposition~\ref{prop:each edge once in minimal sequence}, we have $\sig_{j_i}(\ralpha)\neq\sig_{j_i}(\rbeta)$. Then it must be the case that there exists $1\leq m\leq k$ such that $l_m=j_i$. We can conclude that  $\{j_1,...j_k\}=\{l_1,...,l_k\}$.
\end{proof}

Graphically, Propositions~\ref{prop:each edge once in minimal sequence} and~\ref{prop:same set in minimal sequence} imply that each geodesic between two reduced expressions $\ralpha$ and $\rbeta$ utilizes the same set of edge labels and each label appears once. Instead of labeling edges with the corresponding braid shadow location, we will often assign colors to edges in a braid graph to represent which braid shadow that edge corresponds to.

\begin{example}\label{ex:braid colors new}
Figure~\ref{fig:geocolor new} depicts the braid graph for the link $\rgamma_4= \textcolor{magenta}{343}1\textcolor{turq}{232}43$ from Example~\ref{ex:braid classes}(c). The \textcolor{nectarine}{orange} edges correspond to $b^4$, the \textcolor{turq}{teal} edges correspond to $b^3$, the \textcolor{ggreen}{green} edges correspond to $b^2$, and finally the \textcolor{magenta}{magenta} edges correspond to $b^1$. One can verify that each geodesic between any pair of vertices utilizes the same set of colors with each color appearing exactly once.
\end{example}

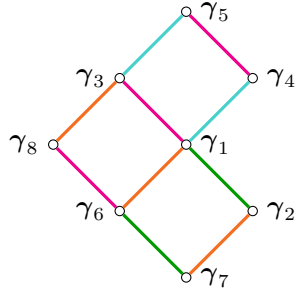
\begin{figure}[ht!]
	\centering
	\begin{tikzpicture}[every circle node/.style={draw, circle, inner sep=1.25pt},scale=1.25,rotate=-45]
		\node [circle] (a) [label=left:$\rgamma_8$] at (0,0){};
		\node [circle] (b) [label=left:$\rgamma_3$] at (0,1){};
		\node [circle] (c) [label=left:$\rgamma_6$] at (1,0){};
		\node [circle] (d) [label=right:$\rgamma_1$] at (1,1){};
		
		\node [circle] (e) [label=right:$\rgamma_5$] at (0,2){};
		\node [circle] (f) [label=right:$\rgamma_4$] at (1,2){};
		\node [circle] (g) [label=right:$\rgamma_2$] at (2,1){};
		\node [circle] (h) [label=right:$\rgamma_7$] at (2,0){};
		
		\draw [nectarine,-, very thick] (a) to (b);
		\draw [magenta,-, very thick] (b) to (d);
		\draw [nectarine,-, very thick] (c) to (d);
		\draw [magenta,-, very thick] (c) to (a);
		\draw [turq,-, very thick] (b) to (e);
		\draw [magenta,-, very thick] (e) to (f);
		\draw [turq,-, very thick] (f) to (d);
		\draw [ggreen,-, very thick] (d) to (g);
		\draw [nectarine,-, very thick] (g) to (h);
		\draw [ggreen,-, very thick] (h) to (c);
	\end{tikzpicture}
	\caption{Braid graph for the link in Example~\ref{ex:braid colors new}. Edges are colored according to which braid shadow they correspond to.} \label{fig:geocolor new}
\end{figure}

For braid equivalent reduced expressions $\ralpha$ and $\rbeta$, we define $\Delta(\sig(\ralpha), \sig(\rbeta))$ to be the number of generators that differ between the signatures of $\ralpha$ and $\rbeta$. The following corollaries are immediate consequences of Propositions~\ref{prop:each edge once in minimal sequence} and~\ref{prop:same set in minimal sequence}. 

\begin{corollary}\label{cor:distisdelta}
If $(W,S)$ is type $\Lambda$ and $\ralpha$ and $\rbeta$ are braid equivalent reduced expressions, then $d(\ralpha,\rbeta)=\Delta(\sig(\ralpha),\sig(\rbeta))$.
\end{corollary}

The signature of a reduced expression of dimension $r$ is a sequence consisting of $r$ generators. Thus, we get the following.

\begin{corollary}\label{diamlrank}
If $(W,S)$ is type $\Lambda$ and $\ralpha$ is a reduced expression, then $\diam(\B(\ralpha)) \leq \dim(\ralpha)$.
\end{corollary}

We conjecture the following.

\begin{conjecture}\label{conj:diamisrank}
If $(W,S)$ is type $\Lambda$ and $\ralpha$ is a reduced expression, then $\diam(\B(\ralpha))=\dim(\ralpha)$.
\end{conjecture}

\section{Partial cube structure of braid graphs}\label{sec:partial cube structure}

The next result states that the reduced expressions in a semicube of a braid graph for a link in Coxeter systems of type $\Lambda$ are uniquely determined by their signature. Moreover, every $F$-class of edges is equal to the set of edges with the same label.

\begin{proposition}\label{prop:same color semis}
If $(W,S)$ is type $\Lambda$, $\ralpha$ is a reduced expression of dimension at least one, and $\{\ralpha,\rbeta\}$ is an edge in $\B(\ralpha)$, then there exists a unique $i$ such that $\sig_i(\ralpha)\neq\sig_i(\rbeta)$ and $W_{\ralpha\rbeta}=\barsig_i(\ralpha)$. Moreover, $\{\rx, \ry\} \in F_{\ralpha\rbeta}$ if and only if $\{\rx, \ry\}$ is labeled by $i$.
\end{proposition}

\begin{proof}
Let $\ralpha$ be a reduced expression of dimension $r\geq1$ and let $\{\ralpha,\rbeta\}$ be an edge of $\B(\ralpha)$. By Corollary~\ref{cor:distisdelta}, there exists a unique $i$ such that $\sig_i(\ralpha)\neq\sig_i(\rbeta)$. Now, we will show that $W_{\ralpha\rbeta}=\barsig_i(\ralpha)$. 

For the forward containment, suppose $\rx\in W_{\ralpha\rbeta}$. For sake of a contradiction, suppose $\sig_i(\rx)\neq\sig_i(\ralpha)$. Using this fact together with Corollary~\ref{cor:distisdelta}, we see that 
\[
d(\rx,\ralpha)=\Delta(\sig(\rx),\sig(\ralpha))=\Delta(\sig(\rx),\sig(\rbeta))+1=d(\rx,\rbeta)+1.
\]
This contradicts Proposition~\ref{prop:semicubeonestepcloser}. Hence, $\sig_i(\rx)=\sig_i(\ralpha)$ and $W_{\ralpha\rbeta}\subseteq \barsig_i(\ralpha)$.

For the reverse containment, suppose $\rx\in\barsig_i(\ralpha)$. Then 
\[
d(\rx,\rbeta)=\Delta(\sig(\rx),\sig(\rbeta))=\Delta(\sig(\rx),\sig(\ralpha))+1=d(\rx,\ralpha)+1.
\]
 Thus, $\rx\in W_{\ralpha\rbeta}$ by Proposition~\ref{prop:semicubeonestepcloser}. This implies that $\barsig_i(\ralpha)\subseteq W_{\ralpha\rbeta}$. 

The statement about $F$-edges now follows immediately.
\end{proof}

\begin{example}\label{ex:semicubesigclass}
Consider the braid graph shown in Figure~\ref{fig:BoxWithTwoFlaps} (see Example~\ref{ex:braid graph with partition}). The semicubes $W_{\rsigma \rtau}=\barsig_5(\rsigma)$ and $W_{\rtau \rsigma}=\barsig_5(\rtau)$ have been highlighted in \textcolor{magenta}{magenta} and \textcolor{turq}{teal}, respectively. Moreover, the black edges connecting $W_{\sigma \rtau}$ and $W_{\rtau \rsigma}$ constitute the equivalence class of $F$-edges containing $\{\rsigma,\rtau\}$, all of which correspond to the fifth braid shadow, $\llb 9,11\rrb$.
\end{example}

One consequence of Proposition~\ref{prop:same color semis} is that the Djoković--Winkler relation $\btheta$ is transitive, and hence an equivalence relation. Since every braid graph is bipartite (Proposition~\ref{braidbi}) and $\btheta$ is an equivalence relation, Proposition~\ref{prop:TFAE partial cube} implies that every braid graph for a reduced expression in a Coxeter system of type $\Lambda$ is a partial cube. This provides an alternate proof of the main result in~\cite{ABCE2024}. Additionally, the authors of~\cite{ABCE2024} conjectured that the isometric dimension of the braid graph is equal to the dimension of the corresponding link. This conjecture now follows from Propositions~\ref{prop:isometric dimension} and~\ref{prop:same color semis}, and so we get the following.

\begin{proposition}\label{prop:dim equals rank}
If $(W,S)$ is type $\Lambda$ and $\ralpha$ is a reduced expression, then $\B(\ralpha)$ is a partial cube with $\dim_I(\B(\ralpha)) = \dim(\ralpha)$.
\end{proposition}

If Conjecture~\ref{conj:diamisrank} is true, then Proposition~\ref{prop:dim equals rank} would imply the following.

\begin{conjecture}\label{conj:iso dim is diam}
If $(W,S)$ is type $\Lambda$ and $\ralpha$ is a reduced expression, then $\dim_I(\B(\ralpha))=\diam(\B(\ralpha))$.
\end{conjecture}

\begin{example}\label{dimrankdiam}
Recall the braid class for the link $\rbeta_1$ in the Coxeter system type $D_4$ from Example~\ref{ex:braid classes}(b). Notice that $\dim(\rbeta_1) = 3$. Figure~\ref{fig:embedding} depicts an isometric embedding of $\B(\rbeta_1)$ into a hypercube of dimension $3$. It is clear that we cannot embed into a hypercube of lower dimension, and so $\dim_I(\B(\rbeta_1)) = 3$, which illustrates Proposition~\ref{prop:dim equals rank}. Moreover, $\diam(\B(\rbeta_1)) = 3$, which confirms Conjectures~\ref{conj:diamisrank} and~\ref{conj:iso dim is diam} in this example.

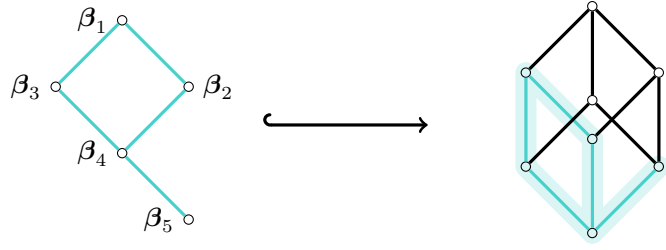
\begin{figure}[h!]
\centering
\begin{tikzpicture}[every circle node/.style={draw, circle ,inner sep=1.25pt}, scale = 1.25]
\begin{scope}[shift={(1,-.15)}] 
\node [circle] (1) [label=left:$\rbeta_4$] at (0,1){};
\node [circle] (2) [label=left:$\rbeta_5$] at (.707,.293){};
\node [circle] (3) [label=left:$\rbeta_3$] at (-.707,1.707){};
\node [circle] (4) [label=right:$\rbeta_2$] at (.707,1.707){};
\node [circle] (5) [label=left:$\rbeta_1$] at (0,2.414){};
\draw [turq,-, very thick] (1) to (2);
\draw [turq,-, very thick] (1) to (3);
\draw [turq,-, very thick] (1) to (4);
\draw [turq,-, very thick] (3) to (5);
\draw [turq,-, very thick] (5) to (4);
\end{scope}
\node [circle] (6) at (6,1){};
\node [circle] (7) [label=left:\phantom{$\scriptstyle 101$}] at (5.293,1.707){};
\node [circle] (8) [label=right:\phantom{$\scriptstyle  011$}] at (6.707,1.707){};
\node [circle] (9) [label=above:\phantom{$\scriptstyle  111$}] at (6,2.414){};
\node [circle] (10) [label=below:\phantom{$\scriptstyle 000$}] at (6,0){};
\node [circle] (11) [label=left:\phantom{$\scriptstyle 100$}] at (5.293,0.707){};
\node [circle] (12) [label=right:\phantom{$\scriptstyle 010$}] at (6.707,0.707){};
\node [circle] (13) at (6,1.414){};
\node at (6-.25,1-.2) {\phantom{$\scriptstyle 001$}};
\node at (6+.25,1.414+.2) {\phantom{$\scriptstyle 110$}};
\draw [turq,-, very thick] (6) to (7);
\draw [black,-, very thick] (6) to (8);
\draw [black,-, very thick] (7) to (9);
\draw [black,-, very thick] (9) to (8);
\draw [turq,-, very thick] (10) to (11);
\draw [turq,-, very thick] (10) to (12);
\draw [black,-, very thick] (11) to (13);
\draw [black,-, very thick] (13) to (12);
\draw [black,-, very thick] (13) to (9);
\draw [turq,-, very thick] (6) to (10);
\draw [turq,-, very thick] (7) to (11);
\draw [black,-, very thick] (8) to (12);
\begin{pgfonlayer}{background}
\highlight{8pt}{turq}{(6.center) to (7.center) to (11.center) to (10.center) to (12.center) to (10.center) to (6.center)}
\end{pgfonlayer}

\draw [very thick, right hook->, black] (2.5,1.15) -- (4.25,1.15) node[midway,above]{};

\end{tikzpicture}
\caption{A braid graph as a partial cube with isometric dimension, dimension, and diameter equal to 3 as described in Example~\ref{dimrankdiam}.}\label{fig:embedding}
\end{figure}
\end{example}

The next result follows immediately from Propositions~\ref{prop:TFAE partial cube}, \ref{prop:same color semis}, and~\ref{prop:dim equals rank}.

\begin{corollary}\label{cor:barsig convex}
If $(W,S)$ is type $\Lambda$ and $\ralpha$ is a reduced expression of dimension at least one, then $\barsig_i(\ralpha)$ is convex.
\end{corollary}

We now state a special case of the previous corollary, which can be viewed as a strengthening of Parts~(e) and (g) of  Proposition~\ref{prop:multiple results cobbled}.

\begin{corollary}\label{convexmoon}
If $(W,S)$ is type $\Lambda$, $\ralpha$ is a link of dimension $r\geq 2$, and we choose $\rsigma \in [\ralpha]$ such that $\llb 2r-3,2r-1 \rrb, \llb 2r-1,2r+1 \rrb \in \mathcal{S}(\rsigma)$ according to Proposition~\ref{prop:special link}, then $\barsig_r(\rsigma)$ and $\barsig_r(b^r(\rsigma))$ are convex. 
\end{corollary}

\section{Cycle structure of braid graphs}\label{sec:cycle structure}

This section contains several new results pertaining to the cycle structure of braid graphs. Recall that every cycle in a Matsumoto graph is of even length. Consequently, every cycle in a braid graph is of even length. A pair of edges in a cycle of even length are called \emph{opposite} if they are opposite when the cycle is interpreted as a regular polygon. Figure~\ref{fig:geocolor new} provides a nice visualization of the next result.

\begin{proposition}\label{prop:opposite edges in 4-cycle}
If $(W,S)$ is type $\Lambda$ and $\ralpha$ is a reduced expression such that $\B(\ralpha)$ contains a 4-cycle, then opposite edges in that 4-cycle correspond to the same braid move.
\end{proposition}

\begin{proof}
Any two opposite edges in a $4$-cycle belong to the same $F$-class. Thus, the result follows from Proposition~\ref{prop:same color semis}.
\end{proof}

The next four results investigate when a 4-cycle appears in a braid graph.

\begin{proposition}\label{when4}
 Suppose $(W,S)$ is type $\Lambda$ and let $\ralpha$ be a reduced expression. Then there exists $\rbeta \in [\ralpha]$ such that $\rbeta$ has two disjoint braid shadows $\llb 2i-1,2i+1 \rrb$ and $\llb 2j-1, 2j+1 \rrb$ if and only if $\B(\ralpha)$ has a 4-cycle with the opposite edges labeled by $i$ and $j$.
\end{proposition}

\begin{proof}
The forward implication is clear. Now, suppose that $\B(\ralpha)$ has a 4-cycle. By Proposition~\ref{prop:opposite edges in 4-cycle}, we may assume that the pairs of opposite edges of this cycle are labeled by some $i$ and $j$ with $i\neq j$. Let $\rbeta$ be any vertex on this cycle. Then $\rbeta$ has two distinct braid shadows $\llb 2i-1,2i+1 \rrb$ and $\llb 2j-1, 2j+1 \rrb$. It remains to show that $\llb 2i-1, 2i+1 \rrb$ and $\llb 2j-1, 2j+1 \rrb$ are disjoint. For sake of contradiction, suppose that $\llb 2i-1, 2i+1 \rrb$ and $\llb 2j-1, 2j+1 \rrb$ intersect, and without loss of generality suppose that $i+1 = j$. Further, suppose that $\rbeta_{\llb 2i-1, 2i+1 \rrb} = sts$ and $\rbeta_{\llb 2i+1, 2i+3 \rrb} = sus$ where $m(s,t) = 3 = m(s,u)$ and $m(t,u) = 2$. Then $b^i(\rbeta)_{\llb 2i+1, 2i+3 \rrb} = tus$, so that $\llb 2j-1,2j+1\rrb= \llb 2i+1, 2i+3 \rrb \notin \mathcal{S}(b^i(\rbeta))$. This is a contradiction since $b^i(\rbeta)$ lies on this 4-cycle.
\end{proof}

\begin{proposition}\label{deg34cyc}
If $(W,S)$ is type $\Lambda$ and $\ralpha$ is a reduced expression such that $\B(\ralpha)$ has a vertex $\rlambda$ of degree 3 or more, then $\rlambda$ is contained in a $4$-cycle.
\end{proposition}

\begin{proof}
Let $\rlambda$ correspond to a vertex of degree 3 or more. Then there exists $\llb i-1, i+1 \rrb, \llb j-1, j+1 \rrb, \llb k-1, k+1 \rrb \in \mathcal{S}(\rlambda)$. It is clear that in any arrangement of these three braid moves, there are at least two that are disjoint, which yields a 4-cycle by Proposition~\ref{when4}.
\end{proof}

The next result is immediate from Proposition~\ref{deg34cyc}.

\begin{corollary}\label{treepath}
If $(W,S)$ is type $\Lambda$ and $\ralpha$ is a reduced expression such that $\B(\ralpha)$ is a tree, then $\B(\ralpha)$ is a path.
\end{corollary}

The following proposition extends the result in Proposition~\ref{prop:opposite edges in 4-cycle} and is a stepping stone to Proposition~\ref{prim4}. 

\begin{proposition}\label{primopp}
If $(W,S)$ is type $\Lambda$ and $\ralpha$ is a reduced expression, then the opposite edges of an even-length isometric cycle in $\B(\ralpha)$ correspond to the same braid move.
\end{proposition}

\begin{proof}
Let $C$ be the subgraph of $B(\ralpha)$ induced by the vertices of an isometric cycle in $\B(\ralpha)$ and let $\{\ralpha, \rbeta\}$ and $\{\ralpha', \rbeta'\}$ be opposite edges of $C$ such that $d(\ralpha, \ralpha') < d(\ralpha, \rbeta')$.

Since $C$ is an isometric cycle consisting of an even number of edges, both the path from $\ralpha$ to $\rbeta'$ passing through $\ralpha'$ along $C$ and the path from $\rbeta$ to $\ralpha'$ passing through $\ralpha$ along $C$ are geodesics. In light of Proposition~\ref{prop:each edge once in minimal sequence}, it must be the case that $\{\ralpha, \rbeta\}$ and $\{\ralpha', \rbeta'\}$ have the same label.
\end{proof}

\begin{proposition}\label{prim4}
If $(W,S)$ is type $\Lambda$ and $\ralpha$ is a reduced expression, then every convex cycle in $\B(\ralpha)$ is of length 4 such that the pairs of opposite edges correspond to disjoint braid shadows.
\end{proposition}

\begin{proof}
Let $C$ be the subgraph of $B(\ralpha)$ induced by the vertices of a convex cycle in $\B(\ralpha)$ and let $\{\ralpha, \rbeta\}$ and $\{\ralpha', \rbeta'\}$ be opposite edges such that $d(\ralpha,\ralpha')<d(\ralpha,\rbeta')$. Suppose $\{\ralpha, \rbeta\}$ is labeled with $i$ and the other edge incident to $\ralpha$ on $C$ is labeled with $j$, so that $\llb 2i-1, 2i+1 \rrb, \llb 2j-1, 2j+1 \rrb \in \bs(\ralpha)$. By Proposition~\ref{primopp}, $\{\ralpha', \rbeta'\}$ is also labeled with $i$. If $\llb 2i-1, 2i+1 \rrb$ and $\llb 2j-1, 2j+1 \rrb$ are disjoint, then $C$ must be a cycle of length 4 by Proposition~\ref{when4} since $C$ is a convex cycle.

Towards a contradiction, suppose $\llb 2i-1, 2i+1 \rrb$ and $\llb 2j-1, 2j+1 \rrb$ overlap, and without loss of generality, assume that $i+1 = j$. Then by Proposition~\ref{prop:local structure of overlapping braid shadows}, $\ralpha$ is of the form
\[
\underbrace{\cdots\lfrac{t}{2i-1}{\lfrac{s}{2i}}\lfrac{t}{2i+1}{\lfrac{u}{2i+2}}\lfrac{t}{2i+3}\cdots}_{\ralpha},
\]
where $m(s,t) = 3 = m(t,u)$ and $m(s,u) = 2$.
Then $b^j(\ralpha)$ is of the form
\[
\underbrace{\cdots\lfrac{t}{2i-1}{\lfrac{s}{2i}}\lfrac{u}{2i+1}{\lfrac{t}{2i+2}}\lfrac{u}{2i+3}\cdots}_{b^j(\ralpha)},
\]
so that $\llb 2i-1,2i+1\rrb \notin \bs(b^j(\ralpha))$. By Proposition~\ref{prop:each edge once in minimal sequence}, there is not another edge labeled with $j$ along the geodesic between $\ralpha$ and $\ralpha'$ passing through $b^j(\ralpha)$. Given the form of $b^j(\ralpha)$, this implies that $\llb 2i-1, 2i+1 \rrb\notin \bs(\ralpha')$, which contradicts $\{\ralpha', \rbeta'\}$ being labeled with $i$. The result now follows.
\end{proof}

\section{Median structure of braid graphs}\label{sec:median structure}

The next result together with Corollary~\ref{cor:braid graph for reduced exp median} constitute the main result of this paper.  

\begin{theorem}\label{thm:braid graph for link median}
If $(W,S)$ is type $\Lambda$ and $\ralpha$ is a link, then $\B(\ralpha)$ is median.
\end{theorem}

\begin{proof}
We will proceed by induction on the dimension $r$ of $\ralpha$. If $\dim(\ralpha) = 0$, then $\B(\ralpha)$ is a single vertex, which is clearly median. If $\dim(\ralpha)=1$, then $\B(\ralpha)$ consists of two vertices connected by a single edge. Since this graph is a peripheral expansion of a single vertex, $\B(\ralpha)$ is median by Mulder's Theorem (Proposition~\ref{prop:mulder}). This verifies the base cases. 

Now, let $\ralpha$ be a link of dimension $r\geq2$ and assume that every braid graph for a link of dimension $r-1$ is median. According to Proposition~\ref{prop:special link}, choose $\rsigma \in [\ralpha]$ such that $ \llb 2r-3,2r-1 \rrb, \llb 2r-1,2r+1 \rrb \in \mathcal{S}(\rsigma)$. Using Proposition~\ref{prop:local structure of overlapping braid shadows}, we may assume that $\rsigma_{\llb 2r-3, 2r+1 \rrb} = tutst$, where $m(u,t) = 3 = m(s,t)$ and $m(u,s) = 2$. By Proposition~\ref{prop:multiple results cobbled}, $\hat{\rsigma}$ is a link of dimension $r-1$ and $\B(\hat{\rsigma})$ is isomorphic to $\B(\ralpha)[\barsig_r(\rsigma)]$. By induction, $\B(\hat{\rsigma})$ is median and as a result, $\B(\ralpha)[\barsig_r(\rsigma)]$ is also median. Hence, by Proposition~\ref{prop:mulder}, it suffices to show that $\B(\ralpha)$ can be obtained from $\B(\ralpha)[\barsig_r(\rsigma)]$ via a single peripheral expansion.

Consider the set
\[
U:=U_{\rsigma b^r(\rsigma)}=\{\rbeta \in \barsig_r(\rsigma)\mid \rbeta\text{ is adjacent to a vertex in }\barsig_r(b^r(\rsigma))\}.
\]
By Proposition~\ref{prop:local structure of braid shadows}, we have $\rbeta_{\llb 2r, 2r+1 \rrb} = st$ for all $\rbeta \in \barsig_r(\rsigma)$ and by Proposition~\ref{prop:multiple results cobbled} we have $\rgamma_{\llb 2r-1, 2r+1 \rrb} = sts$ for all $\rgamma \in \barsig_r(b^r(\rsigma))$. This implies that $U = \{\rbeta \in \barsig_r(\rsigma) \mid \rbeta_{\llb 2r-1, 2r+1 \rrb} = tst\}$, so that if $\rbeta \in U$, then $\llb 2r-1, 2r+1 \rrb \in \mathcal{S}(\rbeta)$. We argue that $U$ is convex. Let $\rgamma_1, \rgamma_2 \in U$. Toward a contradiction, suppose that there exists a geodesic between $\rgamma_1$ and $\rgamma_2$ that includes a vertex $\rkappa \in \barsig_r(\rsigma) \setminus U$. By the preceding argument, $\rgamma_1$ and $\rgamma_2$ end in $tst$ while $\rkappa$ ends in $ust$. Thus, the braid move $b^{r-1}$ must occur at least twice in the geodesic. However, this contradicts Proposition~\ref{prop:each edge once in minimal sequence}. Therefore, $U$ is convex.

Now, consider the set
\[
U_{b^r(\rsigma) \rsigma}=\{\rbeta \in \barsig_r(b^r(\rsigma))\mid \rbeta\text{ is adjacent to a vertex in }\barsig_r(\rsigma)\}.
\]
In this case, we have that $U_{b^r(\rsigma) \rsigma}$ is actually equal to the entirety of $\barsig_r(b^r(\rsigma))$. According to Proposition~\ref{prop:iso for U sets}, the corresponding $F$-edges labeled by $b^r$ induce an isomorphism between the induced subgraphs $\B(\ralpha)[U]$ and $\B(\ralpha)[\barsig_r(b^r(\rsigma))]$.
 It follows that $\B(\ralpha)$ is obtained by a peripheral expansion from the median graph $\B(\ralpha)[\barsig_r(\rsigma)]$. This completes the proof.
\end{proof}

The above theorem combined with Propositions~\ref{medprodmed} and~\ref{prop:braid classes simply laced}\ref{box prod} yields the following.

\begin{corollary}\label{cor:braid graph for reduced exp median}
If $(W,S)$ is type $\Lambda$ and $\ralpha$ is a reduced expression, then $\B(\ralpha)$ is median.
\end{corollary}

As the next example illustrates, not every median graph can be realized as a braid graph in a Coxeter system of type $\Lambda$.

\begin{example}\label{ex:mednotbraid}
Consider the graph $G$ in Figure~\ref{fig:squarestouchingtips}. It is easy to see that $G$ can be obtained from a single vertex via a sequence of peripheral expansions of length four. We have colored the edges of $G$ in order to distinguish the equivalence classes for the Djoković–Winkler relation. By Proposition~\ref{prop:isometric dimension}, we see that $\dim_I(G) = 4$. For sake of contradiction, suppose that $\ralpha$ is a reduced expression with braid graph $G$. Since $G$ is indecomposable with respect to $\square$, we can assume that $\ralpha$ is a link. Then by Proposition~\ref{prop:dim equals rank}, we must have $\dim(\ralpha) = 4$. Since $G$ has a degree four vertex, there exists $\rphi \in [\ralpha]$ such that $|\bs(\rphi)| = 4$. According to~\cite{ABCE2024}, $\rphi$ is a so-called Fibonacci link and thus $G$ must be a Fibonacci cube. It turns out that all Fibonacci cubes have a Fibonacci number of vertices, but $G$ does not, a contradiction. Therefore, $G$ is not the braid graph for a reduced expression in a Coxeter system of type $\Lambda$. 
\end{example}

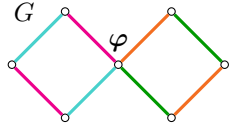
\begin{figure}[h!]
\centering
\begin{tikzpicture}[rotate=90,every circle node/.style={draw, circle, inner sep=1pt}, scale = 1]
\node [circle] [label=above:$\rphi$] (1) at (0,1){};
\node [circle] (2) at (.707,.293){};
\node [circle] (3) at (-.707,1.707){};
\node [circle] (4) at (.707,1.707){};
\node [circle] [](5) at (0,2.414){};
\node [circle] (6) at (-.707,.293){};
\node [circle] [] (7) at (0,-0.414){};
\node (x) at (.7,2.25) {$G$};
\draw [nectarine,-, very thick] (1) to (2);
\draw [turq,-, very thick] (1) to (3);
\draw [magenta,-, very thick] (1) to (4);
\draw [ggreen,-, very thick] (1) to (6);
\draw [magenta,-, very thick] (3) to (5);
\draw [nectarine,-, very thick] (6) to (7);
\draw [turq,-, very thick] (5) to (4);
\draw [ggreen,-, very thick] (7) to (2);
\begin{pgfonlayer}{background}
\end{pgfonlayer}
\end{tikzpicture}
\caption{A median graph that does not arise as a braid graph in a Coxeter system of type $\Lambda$ as described in Example~\ref{ex:mednotbraid}.} \label{fig:squarestouchingtips}
\end{figure}

Corollary~\ref{cor:braid graph for reduced exp median} states that the braid graph of any reduced expression in a Coxeter system of type $\Lambda$ is median. But perhaps braid graphs in all simply-laced Coxeter systems, even those that are not triangle free, are median as the next example suggests.

\begin{example}\label{ex:trimed}
Consider the reduced expression $\rdelta = 1213121$ in the Coxeter system of type $\widetilde A_2$. The braid graph $\B(\rdelta)$ is shown in Figure~\ref{fig:affinebraid}. This graph is median by Proposition~\ref{prop:mulder}. However, the Coxeter system of type $\widetilde A_2$ is not triangle free.
\end{example}

\begin{figure}[ht!]
\centering
\UOLaugment\overline 
\UOLaugment\underline
 \begin{tikzpicture}[every circle node/.style={draw, circle ,inner sep=1.25pt}, scale=1]
  \node [circle] (1) [] at (0,0) {};
  \node [circle] (2) [] at (1,0) {};
  \node [circle] (3) [] at (1 + 0.70711,0.70711) {};
  \node [circle] (4) [] at (1 + 0.70711,-0.70711) {};
  \node [circle] (5) [] at (1 + 2 * 0.70711,0){};
  \node [] (a)[label=right:{\hspace{-10pt}${\rdelta =\underline{12}[1]\overline{3}[1]\underline{21}}$} ] at (1.5 + 2 * 0.70711,.5){};
  \node [] (b)[] at (1 + 2 * 0.70711,.05){};
  \node [circle] (6) [] at (2 + 2 * 0.70711,0){};
  \draw [turq,-, very thick] (1) to (2);
  \draw [turq,-, very thick] (2) to (3);
  \draw [turq,-, very thick] (2) to (4);
  \draw [turq,-, very thick] (3) to (5);
  \draw [turq,-, very thick] (4) to (5);
  \draw [turq,-, very thick] (5) to (6);
  \draw [->, thick] (a) to [out=180,in=90] (b);
\end{tikzpicture}
\caption{The braid graph of $\rdelta = 1213121$ in the Coxeter system of type $\widetilde A_2$ as described in Example~\ref{ex:trimed}.}
\label{fig:affinebraid}
\end{figure}
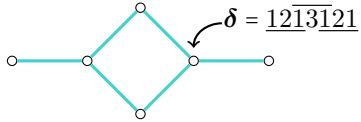

We need an additional definition to proceed. If $\ralpha$ and $\rbeta$ are braid equivalent links, we define \[\barsig(\ralpha,\rbeta):=\{\rx\in[\ralpha]\mid\sig_i(\rx)=\sig_i(\ralpha)\ \text{whenever}\ \sig_i(\ralpha)=\sig_i(\rbeta)\}.\] 
That is, $\barsig(\ralpha,\rbeta)$ is the set of reduced expressions whose signature entries agree with the common signature entries of $\ralpha$ and $\rbeta$. The next proposition states that the reduced expressions on any geodesic between a pair of braid equivalent links share the common signature values of the pair. 

\begin{proposition}\label{prop:sigbar equals interval}
If $(W,S)$ is type $\Lambda$ and $\ralpha$ and $\rbeta$ are braid equivalent reduced expressions, then $I(\ralpha,\rbeta)=\barsig(\ralpha,\rbeta)$.
\end{proposition}

\begin{proof}
Let $\ralpha$ and $\rbeta$ be braid equivalent reduced expressions. The inclusion $I(\ralpha,\rbeta)\subseteq\barsig(\ralpha,\rbeta)$ follows immediately from Proposition~\ref{prop:same set in minimal sequence}.

For the reverse containment, suppose $\rx\in \barsig(\ralpha,\rbeta)$. There is a minimal braid sequence $b_1^{j_1},\ldots, b_i^{j_i}$ from $\ralpha$ to $\rx$ and a minimal braid sequence $b_{i+1}^{j_{i+1}}, \ldots,b_k^{j_k}$ from $\rx$ to $\rbeta$. We claim that $b_1^{j_1},\ldots, b_i^{j_i},b_{i+1}^{j_{i+1}}, \ldots,b_k^{j_k}$ is a minimal braid sequence from $\ralpha$ to $\rbeta$.
By Proposition~\ref{prop:each edge once in minimal sequence}, $j_1,\ldots,j_i$ are pairwise distinct and $j_{i+1},\ldots, j_k$ are pairwise distinct. Again by Proposition~\ref{prop:each edge once in minimal sequence}, it suffices to show that each $j_i$ occurs precisely once. Suppose otherwise, so that there exists $j_m = j_n$ for some $1\leq m \leq i$ and $i+1 \leq n \leq k$. Set $j:=j_m = j_n$. Note that $j_m$ and $j_n$ are the only centers equal to $j$. Moreover, it must be the case that $\sig_j(\ralpha)\neq \sig_j(\rx)$ and $\sig_j(\rbeta)\neq \sig_j(\rx)$. This implies that $\sig_j(\ralpha)=\sig_j(\rbeta)$. On other hand, since $\rx \in \barsig(\ralpha,\rbeta)$ and $\sig_j(\ralpha)\neq \sig_j(\rx)$, it must be the case that $\sig_j(\ralpha)\neq \sig_j(\rbeta)$, a contradiction.
\end{proof}

There is an intimate connection between median graphs and the so-called majority rule. Let $\ralpha, \rbeta$, and $\rgamma$ be braid equivalent reduced expressions of dimension $r \geq 1$. We define the $i$th \emph{majority} of $\ralpha, \rbeta, \rgamma$ via
\[
\maj_i(\ralpha,\rbeta,\rgamma):=
\begin{cases}
\sig_i(\ralpha),\ \text{if}\ \sig_i(\ralpha)=\sig_i(\rbeta)\ \text{or}\ \sig_i(\ralpha)=\sig_i(\rgamma)\\
\sig_i(\rbeta),\ \text{otherwise}.
\end{cases}
\]
That is, when at least two of the generators in the $i$th position of a triple of braid equivalent reduced expressions agree, we record that generator. For braid equivalent reduced expressions $\ralpha, \rbeta, \rgamma$ of dimension $r \geq 1$, we define the \emph{majority} of $\ralpha, \rbeta, \rgamma$ via
\[
\maj(\ralpha,\rbeta,\rgamma):=(\maj_1(\ralpha,\rbeta,\rgamma),\ldots,\maj_r(\ralpha,\rbeta,\rgamma)).
\]
The majority of three reduced expressions results in an ordered list of generators. The next result follows immediately from the definitions. It states that the intersection of the intervals between the three pairs among three braid equivalent reduced expressions is the set of reduced expressions in their braid class whose $i$th signature is the $i$th majority.

\begin{lemma}\label{lem:wicked awesome}
If $(W,S)$ is type $\Lambda$ and $\ralpha,\rbeta$, and $\rgamma$ are braid equivalent reduced expressions, then 
\[
\barsig(\ralpha,\rbeta)\cap \barsig(\rbeta,\rgamma) \cap \barsig(\rgamma,\ralpha) =\{\rx \in [\ralpha] \mid \sig(\rx)=\maj(\ralpha, \rbeta, \rgamma)\}.
\]
\end{lemma}

According to Corollary~\ref{cor:braid graph for reduced exp median}, the braid graph for a reduced expression is a median graph. In particular, the median $\med(\ralpha,\rbeta,\rgamma)$ of any three braid equivalent reduced expressions $\ralpha,\rbeta,\rgamma$ exists. The following result connects the concepts of median and majority. In particular, it shows that the majority of three braid equivalent reduced expressions is the signature of a reduced expression.

\begin{proposition}
If $(W,S)$ is type $\Lambda$ and $\ralpha, \rbeta$, and $\rgamma$ are braid equivalent reduced expressions, then $\med(\ralpha, \rbeta, \rgamma)$ is the unique $\rx$ satisfying $\sig(\rx)=\maj(\ralpha,\rbeta,\rgamma)$.
\end{proposition}

\begin{proof}
Let $\ralpha,\rbeta$, and $\rgamma$ be braid equivalent links. By Proposition~\ref{prop:sigbar equals interval} and Lemma~\ref{lem:wicked awesome}, we have
\[
I(\ralpha, \rbeta) \cap I(\rbeta, \rgamma) \cap I( \rgamma,\ralpha)=\barsig(\ralpha,\rbeta)\cap \barsig(\rbeta,\rgamma) \cap \barsig(\rgamma,\ralpha) = \{\rx \in [\ralpha] \mid \sig(\rx)=\maj(\ralpha, \rbeta, \rgamma)\}.
\]
By Theorem~\ref{thm:braid graph for link median}, this set consists of a unique reduced expression $\med(\ralpha,\rbeta,\rgamma)$.
\end{proof}

\begin{example}\label{ex:medianmaj}
Consider the braid equivalent links
\[
\ralpha= 34\underline{131}2\underline{343}54, \rbeta=\overline{434}1\overline{232}4354, \text{ and }\rgamma=\overline{434}13243\overline{545}
\]
in the Coxeter system of type $D_5$. We see that
$\maj(\ralpha,\rbeta,\rgamma)=(3,1,2,4,5)$,
which is the signature of the link $\rx = 43413234354 \in [\ralpha]$. The braid graph $\B(\ralpha)$ is depicted in Figure~\ref{fig:medianmaj}. The intervals $I(\ralpha,\rbeta)$, $I(\rbeta,\rgamma)$, and $I(\ralpha,\rgamma)$ are highlighted in \textcolor{blue}{blue}, \textcolor{red}{red}, and \textcolor{ggreen}{green}, respectively. The intersection of these three intervals contains precisely one reduced expression $\med(\ralpha,\rbeta,\rgamma) = \rx$.
\end{example}

\begin{figure}[h!]
\centering
\begin{tikzpicture}[every circle node/.style={draw, circle, inner sep=1.25pt},scale=1]
\node [circle] (a) [] at (0,0){};
\node [circle] (b) [] at (0,1){};
\node [circle] (c) [label=below:$\rx$] at (1,0){};
\node [circle] (d) [] at (1,1){};

\node [circle] (e) [] at (0,2){};
\node [circle] (f) [label=above: $\ralpha$] at (1,2){};
\node [circle] (g) [] at (2,1){};
\node [circle] (h) [label=below:$\rbeta$] at (2,0){};

\node [circle] (i) [label=below:$\rgamma$] at (-1,0){};
\node [circle] (j) [] at (-1,1){};
\node [circle] (k) [] at (-1,2){};
\draw [turq,-, very thick] (a) to (b);
\draw [turq,-, very thick] (b) to (d);
\draw [turq,-, very thick] (c) to (d);
\draw [turq,-, very thick] (c) to (a);
\draw [turq,-, very thick] (b) to (e);
\draw [turq,-, very thick] (e) to (f);
\draw [turq,-, very thick] (f) to (d);
\draw [turq,-, very thick] (d) to (g);
\draw [turq,-, very thick] (g) to (h);
\draw [turq,-, very thick] (h) to (c);
\draw [turq,-, very thick] (i) to (j);
\draw [turq,-, very thick] (j) to (k);
\draw [turq,-, very thick] (i) to (a);
\draw [turq,-, very thick] (j) to (b);
\draw [turq,-, very thick] (k) to (e);
\begin{pgfonlayer}{background}
\fill[green,opacity=0.2] \convexpath{a,i,j,k,e,f,d,c}{4pt};
\fill[blue,opacity=0.2] \convexpath{g,h,c,d,f,d}{4pt};
\highlight{8pt}{red}{(i.center) to (a.center) to (h.center)};
\end{pgfonlayer}
\end{tikzpicture}
\caption{Example of median computation for the braid graph discussed in Example~\ref{ex:medianmaj}.}\label{fig:medianmaj}
\end{figure}
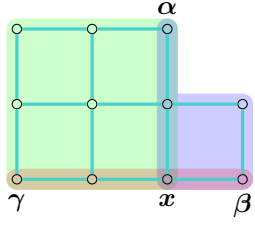

It is well known that median graphs exhibit the Helly Property~\cite{ImrichKlavzarHammack2000}: if $\mathcal{K}$ is an arbitrary family of pairwise-intersecting convex sets, then all sets in $\mathcal{K}$ have a common intersection. In particular, since each $\barsig_i(\ralpha)$ is convex by Corollary~\ref{cor:barsig convex}, we obtain the following. 

\begin{proposition}
If $(W,S)$ is type $\Lambda$, $\ralpha$ is a reduced expression, and $\mathcal{K}$ is a family of pairwise-intersecting sets of the form $\barsig_i(\rbeta)$ for $\rbeta\in[\ralpha]$, then all sets in $\mathcal{K}$ have a common intersection.
\end{proposition}

We believe a related claim is true.

\begin{conjecture}\label{conj:triples of sigbar sets not pairwise disjoint}
If $(W,S)$ is type $\Lambda$ and $\ralpha$ is a reduced expression of dimension $r\geq 1$, then any triple $\{\barsig_i(\rbeta),\barsig_j(\rgamma),\barsig_k(\rsigma)\}$ is not pairwise disjoint, where $\rbeta,\rgamma,\rsigma\in [\ralpha]$ and $1\leq i,j,k\leq r$.
\end{conjecture}

It would be interesting to provide a succinct algebraic description of an arbitrary convex set in a braid graph in a Coxeter system of type $\Lambda$.

\section{Conjectures and Open Problems}\label{sec:closing}

Below we recall the conjectures explicitly mentioned previously and summarize a few additional open problems.

\begin{enumerate}
\item Conjecture~\ref{conj:diamisrank}: If $(W,S)$ is type $\Lambda$ and $\ralpha$ is a reduced expression, then $\diam(\B(\ralpha))=\dim(\ralpha)$. If true, it follows that that if $\ralpha=\ralpha_1 \mid \cdots \mid \ralpha_k$ is the link factorization, then
\[
\diam(\mathcal{B}(\ralpha))=\sum_{i=1}^k\dim(\ralpha_i).
\]

\item Conjecture~\ref{conj:iso dim is diam}: If $(W,S)$ is type $\Lambda$ and $\ralpha$ is a reduced expression, then $\dim_I(\B(\ralpha))=\diam(\B(\ralpha))$. Note that if Conjecture~\ref{conj:diamisrank} is true, then Proposition~\ref{prop:dim equals rank} would imply Conjecture~\ref{conj:iso dim is diam}, giving
\[
\dim_I(\B(\ralpha))=\diam(\B(\ralpha))=\dim(\ralpha).
\]

\item Conjecture~\ref{conj:triples of sigbar sets not pairwise disjoint}: If $(W,S)$ is type $\Lambda$ and $\ralpha$ is a reduced expression of dimension $r\geq 1$, then any triple $\{\barsig_i(\rbeta),\barsig_j(\rgamma),\barsig_k(\rsigma)\}$ is not pairwise disjoint, where $\rbeta,\rgamma,\rsigma\in [\ralpha]$ and $1\leq i,j,k\leq r$.

\item Can we provide a characterization of convex sets in braid graphs in Coxeter systems of type $\Lambda$?

\item Suppose $(W,S)$ is type $\Lambda$ and $\ralpha$ is a link of dimension at least one. If we apply a commutation move to a link in $[\ralpha]$, do we obtain a link (in a new braid class)? What structural properties are preserved between the braid graph for $\ralpha$ and the braid graph for the reduced expression we obtain by applying a commutation move to $\ralpha$?

\item The \emph{geodetic number} of a connected graph $G$ is the minimum number of vertices in a set $S$ whose geodetic closure (i.e., $\bigcup_{u,v\in S}I(u,v)$) is all of $V(G)$. If $(W,S)$ is type $\Lambda$ and $\ralpha$ is a reduced expression, we conjecture that the geodetic number of $\B(\ralpha)$ is two.  Moreover, if $\ralpha$ is a link, we conjecture that there is a unique set of cardinality two whose geodetic closure is all of $[\ralpha]$. This uniqueness claim is certainly not true for arbitrary reduced expressions (see Example~\ref{ex:braid classes}(a)).

\item Related to the previous item, we conjecture that for a link $\ralpha$ in Coxeter systems of type $\Lambda$, there is a unique diametrical pair that determines the diameter of $\B(\ralpha)$. As with the previous item, this claim is not true for arbitrary reduced expressions (see Example~\ref{ex:braid classes}(a)).

\item As seen in Example~\ref{ex:mednotbraid}, not every median graph can be realized as the braid graph for a reduced expression in a Coxeter system of type $\Lambda$. Can we classify the subclass of median graphs that arise as braid graphs in type $\Lambda$ Coxeter systems? 

\item Call a subset $X \subseteq\{0,1\}^n$ \emph{admissible} if there exists a  Coxeter system of type $\Lambda$ and reduced expression $\ralpha$ such that the induced subgraph $Q_n[X]$ is isomorphic to $B(\ralpha)$. Are there necessary and sufficient conditions for $X$ to be admissible?

\item The \emph{Fibonacci dimension} $\dim_F(G)$ of a partial cube $G$ is the smallest integer $d$ such that $G$ can be isometrically embedded into a Fibonacci cube of dimension $d$. Can the Fibonacci dimension of the braid graph for a reduced expression $\ralpha$ in a Coxeter system of type $\Lambda$ be expressed in terms of $\dim(\ralpha)$? By Proposition~\ref{prop:dim equals rank} and~\cite[Proposition 3.1]{cabello2011fibonacci}, we have $\dim(\ralpha) \leq \dim_F(\mathcal{B}(\ralpha)) \leq 2\dim(\ralpha)-1$.

\item As suggested by Example~\ref{ex:trimed}, perhaps braid graphs in \emph{all} simply-laced Coxeter systems are median.  Nearly all of the results in~\cite{ABCE2024} and this paper fundamentally rely on the Coxeter systems being triangle free.  Can we generalize to overcome the three-cycle obstruction?

\item We conjecture that braid graphs for links in all simply-laced Coxeter systems are indecomposable with respect to box product. That is, if a graph $G$ can be decomposed into the box of two nontrivial graphs, then we claim that $G$ is not the braid graph for a link.

\item Can we generalize the notion of braid shadow and link to account for arbitrary $m(s,t)$? We conjecture that most of the known results concerning braid graphs in Coxeter systems of type $\Lambda$ can be generalized to handle braid moves of odd length. However, generalizing to even length braid moves will require substantially more retooling.

\end{enumerate}


\bibliographystyle{plain}
\bibliography{biblio}

\end{document}